\documentclass[a4paper]{article}
\usepackage{amsfonts,amsgen,amstext,amsbsy,amsopn,amsfonts,amssymb,amscd}
\usepackage[leqno]{amsmath}
\usepackage{dsfont}
\usepackage[amsmath,amsthm,thmmarks]{ntheorem}
\usepackage{listings}
\usepackage{epsf,epsfig}
\usepackage{float}
\usepackage{ebezier,eepic}
\usepackage{color}
\usepackage{tikz}
\usepackage{multirow}
\usepackage{mathrsfs}
\usepackage{graphicx}
\usepackage{subfigure}
\newtheorem{thm}{Theorem}[section]

\newtheorem{prop}[thm]{Proposition}

\newtheorem{example}[thm]{Example}
\newtheorem{lem}[thm]{Lemma}
\newtheorem{false statement}{False statement}
\newtheorem{cor}[thm]{Corollary}
\newtheorem{fact}[thm]{Fact}

\theoremstyle{definition}
\newtheorem{defn}[thm]{Definition}
\newtheorem{claim}[thm]{Claim}

\makeatletter \@addtoreset{equation}{section}

\def\hh{\mathcal{H}}
\def\hhat{\widehat{\mathds{H}}}

\def\hht{\mathcal{T}}
\def\hf{\mathcal{F}}
\def\hg{\mathcal{G}}
\def\hk{\mathcal{K}}
\def\ha{\mathcal{A}}
\def\hb{\mathcal{B}}

\def\hp{\mathcal{P}}

\begin{document}

\title{\bf\Large  Intersecting families with covering number three }
\date{}
\author{Peter Frankl$^1$, Jian Wang$^2$\\[10pt]
$^{1}$R\'{e}nyi Institute, Budapest, Hungary\\[6pt]
$^{2}$Department of Mathematics\\
Taiyuan University of Technology\\
Taiyuan 030024, P. R. China\\[6pt]
E-mail:  $^1$frankl.peter@renyi.hu, $^2$wangjian01@tyut.edu.cn
}
\maketitle

\begin{abstract}
We consider $k$-graphs on $n$ vertices, that is, $\mathcal{F}\subset \binom{[n]}{k}$. A $k$-graph $\mathcal{F}$ is called  intersecting  if $F\cap F'\neq \emptyset$ for all $F,F'\in \mathcal{F}$. In the present paper we prove that for $k\geq 7$, $n\geq 2k$, any intersecting $k$-graph $\mathcal{F}$ with covering number at least three, satisfies $|\mathcal{F}|\leq \binom{n-1}{k-1}-\binom{n-k}{k-1}-\binom{n-k-1}{k-1}+\binom{n-2k}{k-1}+\binom{n-k-2}{k-3}+3$,  the best possible upper bound which was proved in \cite{F80} subject to exponential constraints $n>n_0(k)$.

\vspace{6pt}
{\noindent\bf AMS classification:} 05D05.

\vspace{6pt}
{\noindent\bf Key words:} Intersecting families; covering number; shifting; Erd\H{o}s-Ko-Rado Theorem.
\end{abstract}

\section{Introduction}

Let $[n]=\{1,\ldots,n\}$ be the standard $n$-element set, $\binom{[n]}{k}$ the collection of its $k$-subsets, $0\leq k\leq n$.  A subset $\hf\subset \binom{[n]}{k}$ is called a {\it $k$-uniform family} or simply {\it $k$-graph}. The family $\hf$ is said to be {\it intersecting} if $F\cap F'\neq\emptyset$ for all $F,F'\in \hf$. Similarly two families $\hf,\hg$ are called {\it cross-intersecting} if $F\cap G\neq \emptyset$ for all $F\in \hf$, $G\in \hg$. Set $\cap \hf =\mathop{\cap}\limits_{F\in \hf}F$. If $\cap \hf \neq \emptyset$ then $\hf$ is called a {\it star}. Stars are the simplest examples of intersecting families. The quintessential Erd\H{o}s-Ko-Rado Theorem shows that they are the largest as well.

\begin{thm}[\cite{ekr}]
Suppose that $n\geq 2k>0$, $\hf\subset\binom{[n]}{k}$ is intersecting, then
\begin{align}\label{ineq-ekr}
|\hf| \leq \binom{n-1}{k-1}.
\end{align}
\end{thm}

In the case $n=2k$, $\binom{n-1}{k-1}=\frac{1}{2}\binom{n}{k}$ and being intersecting is equivalent to $|\hf\cap\{F,[n]\setminus F\}|\leq 1$ for every $F\in \binom{[n]}{k}$. Consequently there are $2^{\binom{n-1}{k-1}}$ intersecting families $\hf\subset \binom{[2k]}{k}$ attaining equality in \eqref{ineq-ekr}. However for $n>2k\geq 4$ there is a strong stability.

\begin{thm}[Hilton-Milner Theorem \cite{HM67}]
Suppose that $n> 2k\geq 4$, $\hf\subset\binom{[n]}{k}$ is intersecting and $\hf$ is not a star, then
\begin{align}\label{ineq-hm}
|\hf| \leq \binom{n-1}{k-1}-\binom{n-k-1}{k-1}+1.
\end{align}
\end{thm}

Let us define the Hilton-Milner Family
\[
\hh(n,k) = \left\{F\in \binom{[n]}{k}\colon 1\in F,\ F\cap [2,k+1]\neq \emptyset\right\}\cup \{[2,k+1]\},
\]
showing that \eqref{ineq-hm} is best possible.

For an intersecting family
$\hf\subset \binom{[n]}{k}$, define the family of {\it transversals} (or {\it covers}), $\hht(\hf)$ by
\[
\hht(\hf) =\left\{T\subset [n]\colon T\cap F\neq \emptyset \mbox{ for all } F\in \hf\right\}.
\]
Define the {\it covering number} $\tau(\hf) =\min \{|T|\colon T\in \hht(\hf)\}$. For $\tau(\hf)\leq i\leq k$, let
\[
\hht^{(i)}(\hf) =\{T\in \hht(\hf)\colon |T|=i\}.
\]
Note that $\tau(\hf)-1$ is the minimum number of vertices whose deletion results in a star.

Obviously, $\tau(\hf)=1$ iff $\hf$ is a star. The Hilton-Milner Family $\hh(n,k)$ provides the maximum of $|\hf|$ for intersecting $k$-graphs with $\tau(\hf) \geq 2$.

\begin{example}
Define
\[
\hb = \{[2,k+1], \{2\}\cup [k+2,2k], \{3\}\cup [k+2,2k]\}
\]
and
\[
\ha =\left\{A\in \binom{[n]}{k}\colon 1\in A \mbox{ and } A\cap B\neq \emptyset \mbox{ for each }B\in \hb\right\}.
\]
Set $\hg(n,k)=\ha\cup \hb$.
\end{example}

It is easy to check that for $n\geq 2k$, $\hg(n,k)$ is an intersecting $k$-graph with $\tau(\hg)=3$. The 3-element transversals are $\{1,2,3\}$ and  $\{1,u,v\}$ with $u\in [2,k+1]$, $v\in [k+2,2k]$. Consequently, for $k$ fixed and $n\rightarrow \infty$
\begin{align}\label{ineq-5.1}
|\hg(n,k)|=(k^2-k+1)\binom{n-3}{k-3}+O\left(\binom{n-4}{k-4}\right).
\end{align}
One can compute $|\hg(n,k)|$ exactly by using inclusion-exclusion
\begin{align}\label{ineq-5.2}
|\hg(n,k)| = \binom{n-1}{k-1}-\binom{n-k}{k-1}-\binom{n-k-1}{k-1}+\binom{n-2k}{k-1}+\binom{n-k-2}{k-3}+3.
\end{align}

Note that
\[
\frac{\binom{n-k}{k-1}}{\binom{n-1}{k-1}}= \frac{(n-k)(n-k-1)\ldots(n-2k+2)}{(n-1)(n-2)\ldots(n-k+1)}<\left(\frac{n-k}{n-1}\right)^{k-1}
=\left(1-\frac{k-1}{n-1}\right)^{k-1}<e^{-\frac{(k-1)^2}{n-1}}.
\]
If $n,k\rightarrow \infty$ with $k^2/n\rightarrow \infty$ then by \eqref{ineq-5.2},
\[
|\hg(n,k)|=(1-o(1))\binom{n-1}{k-1}.
\]
Let us define the function
\[
f(n,k,s) =\max\left\{|\hf|\colon \hf\subset \binom{[n]}{k} \mbox{ is intersecting and }\tau(\hf)\geq s\right\}.
\]
With this terminology the Erd\H{o}s-Ko-Rado and Hilton-Miner Theorems can be stated as
\[
f(n,k,1) = \binom{n-1}{k-1}, \ f(n,k,2) = \binom{n-1}{k-1}-\binom{n-k-1}{k-1}+1 \mbox{ for } n\geq 2k.
\]
Erd\H{o}s and Lov\'{a}sz \cite{EL} proved that $\lfloor k!(e-1)\rfloor\leq f(n,k,k)\leq k^k$ and $f(n,3,3)=10$. Lov\'{a}sz \cite{lovasz} conjectured that $\lfloor k!(e-1)\rfloor$ is the exact bound. In  \cite{FOT96}, Lov\'{a}sz's conjecture was disproved for $k\geq 4$ and  the lower bound of $f(n,k,k)$ was improved to $\left(\frac{k}{2}+1\right)^{k-1}$ for even $k$ and $\left(\frac{k+3}{2}\right)^{(k-1)/2}\left(\frac{k+1}{2}\right)^{(k-1)/2}$ for odd $k$.

The first author proved in \cite{F80} that for $k\geq 4$, $n>n_0(k)$
\begin{align}\label{ineq-5.3}
f(n,k,3)= |\hg(n,k)|
\end{align}
and up to isomorphism $\hg(n,k)$ is the only optimal family. The methods of that paper can be improved to yield $n_0(k)=ck^2$ for some absolute constant $c$.

We should mention that $f(n,k,4)$ was determined for $k\geq 9$ and $n>n_0(k)$ in \cite{FOT95} and the cases of $k=4,5$ were solved by Chiba, Furuya, Matsubara, Takatou \cite{CFMT} and Furuya, Takatou \cite{FuruyaTakatou}, respectively.

In the present paper, we prove \eqref{ineq-5.3} for  $k\geq 7$ and $n\geq 2k$.

\begin{thm}\label{thm-main3}
For $k\geq 7$ and $n\geq 2k$,
\begin{align}\label{ineq-fnk3}
f(n,k,3)= |\hg(n,k)|.
\end{align}
\end{thm}

This is a considerable improvement on the results of Kupavskii \cite{Ku2} which prove the existence of a large constant $D$ such that \eqref{ineq-fnk3} holds for $k>D$ and $n>Dk$.

An important tool to tackle intersecting families is shifting that can be tracked back to Erd\H{o}s-Ko-Rado \cite{ekr}. We are going to give the formal definition in Section 2 but let us define here the ``end product" of shifting. Let $(x_1,\ldots,x_k)$ denote the set $\{x_1,\ldots,x_k\}$ where we know or want to stress that $x_1<\ldots<x_k$.
Define the {\it shifting partial order} $\prec$ by setting $(a_1,\ldots,a_k)\prec (b_1,\ldots,b_k)$ iff $a_i\leq b_i$ for $1\leq i\leq k$. Then $\ha\subset \binom{[n]}{k}$ is called {\it initial} iff for all  $A,B\in \binom{[n]}{k}$, $A\prec B$ and $B\in \ha$ imply $A\in \ha$.

Let us mention that while the full star and the Hilton-Milner Family are initial, $\hg(n,k)$ is not. Define the function
\[
g(n,k,s) =\max\left\{|\hf|\colon \hf\subset \binom{[n]}{k}\mbox{ is intersecting, initial and }\tau(\hf)\geq s\right\}.
\]
The clique number $\omega(\hf)$ for $\hf\subset {[n]\choose k}$ is defined as
$$\omega(\hf)=\max\left\{q\colon \exists Q\in{[n]\choose q}, {Q\choose k}\subset \hf\right\}.$$
Clearly, $\tau(\hf)\geq \omega(\hf)-k+1$.
For $k\geq s$, define
\[
\hk(n,k,s) = \left\{K\in \binom{[n]}{k}\colon 1\in K, |K\cap [2,k+s-1]|\geq s-1\right\}\bigcup \binom{[2,k+s-1]}{k}.
\]
It is easy to see that $\hk(n,k,s)$ is intersecting, initial and $\tau(\hf)=\omega(\hf)-k+1=s$. Thus $g(n,k,s) \geq |\hk(n,k,s)|$.

Let us show that $g(n,k,s)$ can be deduced from a result in \cite{F19} concerning intersecting families with clique number at least $q$.

\begin{thm}[\cite{F19}]\label{thm-f19}
Let $n>2k\geq 2s$. Suppose that $\hf\subset {[n]\choose k}$ is intersecting and $\omega(\hf)\geq k+s-1$. Then
\begin{align}\label{ineq-f19}
|\hf|\leq |\hk(n,k,s)|.
\end{align}
\end{thm}

Define
\[
\hf(i)=\left\{F\setminus \{i\}\colon i\in F\in \hf\right\},\ \hf(\bar{i})= \{F\colon i\notin F\in \hf\}
\]
and note that $|\hf|=|\hf(i)|+|\hf(\bar{i})|$. For $A\subset B\subset X$, define
\[
\hf(A,B) =\left\{F\setminus B\colon F\in \hf,\ F\cap B=A\right\}.
\]
We also use $\hf(\bar{Q})$ to denote $\hf(\emptyset, Q)$. For $\hf(\{i\},Q)$ we simply write  $\hf(i,Q)$.

The following theorem is an easy consequence of Theorem \ref{thm-f19}.

\begin{thm}\label{thm-1.12}
For $n>2k\geq 2s$,
\begin{align*}
g(n,k,s) =|\hk(n,k,s)|.
\end{align*}
\end{thm}

\begin{proof}
Let $\hf\subset \binom{[n]}{k}$ be intersecting, initial and $\tau(\hf)\geq s$.  Since $\tau(\hf)\geq s$, there exists an edge $F$ in $\hf(\overline{[s-1]})$. Then by initiality we see that $[s,k+s-1]\in \hf$. It follows that $\binom{[k+s-1]}{k}\subset \hf$. That is, $\omega(\hf)\geq k+s-1$. Noting that $\tau( \hk(n, k, s) )=s$, the statement of the theorem follows from Theorem \ref{thm-f19}.
\end{proof}

Two cross-intersecting families $\ha,\hb$ are called {\it non-trivial} if none of them is a star, i.e., $\tau(\ha)\geq 2$ and $\tau(\hb)\geq 2$.  Define $\ha_0 = \{U, V \}$ where $U$ and $V$ are two disjoint $a$-sets in $[n]$. Let
\[
\hb_0 =\left\{B\in\binom{[n]}{b}\colon B\cap U\neq \emptyset,\ B\cap V\neq \emptyset\right\}.
\]
Clearly, $\ha_0, \hb_0$ are non-trivial cross-intersecting.

Recently, the first author \cite{F2022} proved the following result concerning non-trivial cross-intersecting families, which is  essential for the proof of Theorem \ref{thm-main3}.

\begin{thm}[\cite{F2022}]
Let $2\leq a\leq b$, $n\geq a+b$. Suppose that $\ha\subset \binom{[n]}{a}$ and $\hb\subset \binom{[n]}{b}$ are non-trivial and cross-intersecting. Then
\begin{align}\label{ineq-nontrivial}
|\ha|+|\hb| \leq \binom{n}{b}-2\binom{n-a}{b}+\binom{n-2a}{b}+2.
\end{align}
Moreover, if $n>a+b$, then up to symmetry $\ha_0$, $\hb_0$ are the only families
achieving equality in \eqref{ineq-nontrivial}  unless $a=b=2$.
\end{thm}

It is easy to check that \eqref{ineq-nontrivial} also holds for $a=1$ and $b\geq 2$. By using \eqref{ineq-nontrivial} and establishing two analytical inequalities, we prove  the following two extensions of \eqref{ineq-nontrivial}, which are  also needed  in the proof of Theorem \ref{thm-main3}.

\begin{prop}\label{prop-key1}
Let $\ha \subset\binom{[n]}{a}$ and $\hb \subset \binom{[n]}{b}$ be cross-intersecting families, $n\geq a+b$, $b>a\geq 1$. Suppose that $\ha$ is non-trivial and $\hb$ is non-empty. Then
\begin{align}\label{hahbnontrivial}
  |\ha|+|\hb| \leq \binom{n}{b}-2\binom{n-a}{b}+\binom{n-2a}{b}+2.
\end{align}
Moreover, for $n>a+b$ and $a\geq 2$, $\ha_0$, $\hb_0$ are the only families
achieving equality in \eqref{hahbnontrivial}.
\end{prop}

\begin{prop}\label{cor-5.9}
Let $\ha \subset\binom{[n]}{a}$ and $\hb \subset \binom{[n]}{b}$ be cross-intersecting families, $n\geq a+b$, $b\geq a+2\geq 3$. If $\ha$ is non-trivial, then
\begin{align}\label{hahbnontrivial2}
  |\ha|+|\hb| \leq \binom{n}{b}-2\binom{n-a}{b}+\binom{n-2a}{b}+2.
\end{align}
Moreover, for $n>a+b$ and $a\geq 2$, $\ha_0$, $\hb_0$ are the only families
achieving equality in \eqref{hahbnontrivial2}.
\end{prop}

Let us present some more results  that are needed in our proofs. We need the Frankl-Tokushige inequality \cite{FT92, F2022} as follows.

\begin{thm}[\cite{FT92, F2022}]\label{thm-ft92}
Let $\ha \subset\binom{X}{a}$ and $\hb \subset \binom{X}{b}$ be non-empty cross-intersecting families with $n=|X|\geq a+b$, $a\leq b$. Then
\begin{align}\label{ft92}
|\ha|+|\hb|\leq \binom{n}{b} - \binom{n-a}{b}+1.
\end{align}
Moreover, unless $n=a+b$ or $a=b=2$ the inequality is strict for $|\ha|> 1$ and $|\hb|> 1$.
\end{thm}

For $\hf\subset \binom{[n]}{k}$ and $A\subset D\subset [n]$, define $\alpha(A,D) =\frac{|\hf(A,D)|}{\binom{n-|D|}{k-|A|}}$. Note that $D\in \hht(\hf)$ implies $\hf(\emptyset,D)=\emptyset$.

The next statement can be deduced using an old argument of Sperner \cite{Sperner}.

\begin{lem}[\cite{F2022,Sperner}]
Let $\hf\subset \binom{[n]}{k}$ be an intersecting family  and let $D\in \hht(\hf)$. If $A,B\subset D$, $A\cap B=\emptyset$ then
\begin{align}\label{ineq-sperner}
\alpha(A,D)+\alpha(B,D)\leq 1.
\end{align}
\end{lem}

We need also the following three inequalities concerning binomial coefficients.

\begin{prop}[\cite{FW2022}]
Let $n,k,i$ be positive integers. Then
\begin{align} \label{ineq-key1}
&\binom{n-i}{k} \geq \left(\frac{n-k-(i-1)}{n-(i-1)}\right)^i \binom{n}{k}.
\end{align}
\end{prop}

\begin{prop}
For $n\geq 4k$ and $\ell+i\leq 13$,
\begin{align}\label{ineq-key13}
\binom{n-i-\ell}{k-3} \geq \left(\frac{3}{4}\right)^{\ell}\binom{n-i}{k-3}
\end{align}
\end{prop}

\begin{proof}
Note that $\ell+i\leq 13$ implies $n\geq 4k\geq (\ell+i-1)k/3$. It follows that
\[
\frac{n-k-\ell-i+4}{n-\ell-i+1} \geq \frac{n-k}{n}.
\]
 Thus,
\[
\frac{\binom{n-i-\ell}{k-3}}{\binom{n-i}{k-3}}\geq \left(\frac{n-k-\ell-i+4}{n-\ell-i+1}\right)^{\ell} \geq  \left(\frac{n-k}{n}\right)^{\ell}\geq \left(\frac{3}{4}\right)^{\ell}.
\]
\end{proof}


\begin{prop}
Let $n,k,p,i$ be positive integers with $i\geq 3$. Then
\begin{align}\label{ineq-key2}
&\binom{n-i}{k-2}-\binom{n-p-i}{k-2} \geq \left(\frac{n-k-i+3}{n-i+1}\right)^{i-2} \left(\binom{n-2}{k-2}-\binom{n-p-2}{k-2}\right).
\end{align}
\end{prop}

\begin{proof}
Note that
\begin{align*}
\binom{n-i}{k-2}-\binom{n-p-i}{k-2} &= \frac{(n-k)\ldots(n-k-i+3)}{(n-2)\ldots(n-i+1)}\binom{n-2}{k-2}\\[3pt]
&\qquad - \frac{(n-p-k)\ldots(n-p-i-k+3)}{(n-p-2)\ldots(n-p-i+1)} \binom{n-p-2}{k-2}\\[3pt]
&\geq \frac{(n-k)\ldots(n-k-i+3)}{(n-2)\ldots(n-i+1)} \left(\binom{n-2}{k-2}-\binom{n-p-2}{k-2}\right)\\[3pt]
&\geq \left(\frac{n-k-i+3}{n-i+1}\right)^{i-2} \left(\binom{n-2}{k-2}-\binom{n-p-2}{k-2}\right).
\end{align*}
\end{proof}


\section{Shifting ad extremis and 2-cover graphs}

In this section, we define a new technique called shifting ad extremis that is essential for the present paper.

Recall the shifting operation as follows. Let $1\leq i< j\leq n$, $\hf\subset{[n]\choose k}$. Define
$$S_{ij}(\hf)=\left\{S_{ij}(F)\colon F\in\hf\right\},$$
where
$$S_{ij}(F)=\left\{
                \begin{array}{ll}
                  (F\setminus\{j\})\cup\{i\}, & j\in F, i\notin F \text{ and } (F\setminus\{j\})\cup\{i\}\notin \hf; \\[3pt]
                  F, & \hbox{otherwise.}
                \end{array}
              \right.
$$

 Let $\hf\subset\binom{[n]}{k}$, $\hg\subset\binom{[n]}{\ell}$ be families having certain properties (e.g., intersecting, cross $t$-intersecting) that are maintained by simultaneous shifting and certain properties (e.g., $\tau(\hf)\geq s$) that might be destroyed by shifting. Let $\hp$ be the collection of the latter properties that we want to maintain.

Define the quantity
\[
w(\hf) =\sum_{F\in \hf} \sum_{i\in F} i.
\]
Obviously $w(S_{ij}(\hf))\leq w(\hf)$ for $1\leq i<j\leq n$ with strict inequality unless $S_{ij}(\hf)=\hf$.

\begin{defn}\label{defn-2.1}
Suppose that $\hf\subset \binom{[n]}{k}$, $\hg\subset \binom{[n]}{\ell}$ are families having property $\hp$. We say that $\hf$ and $\hg$ have been {\it shifted ad extremis} with respect to $\hp$ if  $S_{ij}(\hf)=\hf$ and $S_{ij}(\hg)=\hg$ for every pair $1\leq i<j\leq n$ whenever $S_{ij}(\hf)$ and $S_{ij}(\hg)$ also have property $\hp$.
\end{defn}

Let $\hf \subset\binom{[n]}{k}$ be an intersecting family and $\hp=\{\tau(\hf)\geq 3\}$. Then $\hf$ is shifted ad extremis if $S_{ij}(\hf)=\hf$ for all $1\leq i<j\leq n$ unless $\tau(S_{ij}(\hf))=2$. A pair $(i,j)$, $1\leq i<j\leq n$, is called a {\it shift-resistant pair} if $\tau(S_{ij}(\hf))=2$. Define the  {\it shift-resistant graph} $\mathds{H}=\mathds{H}_{\hp}(\hf)$ of shift-resistant pairs $(i,j)$ such that
\begin{align*}
& \tau(S_{ij}(\hf))=2 \mbox{ for } (i,j)\in  \mathds{H} \mbox{ and } S_{ij}(\hf)=\hf \mbox{ for }(i,j)\notin \mathds{H}.
\end{align*}
Note that $\hf$ is initial if and only if $\mathds{H}$ is empty.

We can obtain shifted ad extremis families by the following shifting ad extremis process. Let $\hf$, $\hg$ be cross-intersecting families with property $\hp$. Apply the shifting operation $S_{ij}$, $1\leq i<j\leq n$, to $\hf,\hg$ simultaneously and continue as long as the property $\hp$ is maintained. Recall that the shifting operation preserves the cross-intersecting property (cf. \cite{F87}). By abuse of notation, we keep denoting the current families by $\hf$ and $\hg$ during the shifting process. If $S_{ij}(\hf)$ or $S_{ij}(\hg)$ does not have property $\hp$, then we do not apply  $S_{ij}$ and choose a different pair $(i',j')$. However we keep returning to previously failed pairs $(i,j)$, because it might happen that at a later stage in the process $S_{ij}$ does not destroy property $\hp$ any longer. Note that the quantity $w(\hf)+w(\hg)$ is a positive integer and it decreases strictly in  each step. Eventually we shall arrive at families that are shifted ad extremis with respect to $\hp$.

We say that  $\ha\subset\binom{[n]}{\ell},\hb\subset\binom{[n]}{k}$ are {\it saturated cross-intersecting families with respect to $\ha$} if adding an extra $\ell$-set to $\ha$ would destroy the cross-intersecting property.

Recall that $\hht^{(3)}(\hf)$ is the collection of all transversals of size 3 in $\hf$.

\begin{lem}\label{lem2.1}
 Let $\hf\subset \binom{[n]}{k}$ be an intersecting family  of the maximal size with $\tau(\hf)\geq 3$. Suppose that $\hf$ is shifted ad extremis for $\tau(\hf)\geq 3$ with the shift-resistant graph $\mathds{H}$, and  $\hht^{(3)}(\hf)$ is a star with center $a$. Let $\ha=\hf(a)$, $\hb=\hf(\bar{a})$. Then $\ha, \hb$ are saturated cross-intersecting families with respect to $\ha$, and $\hb$ is intersecting. Moreover, $\ha, \hb$ are shifted ad extremis for $\tau(\hb)\geq 2$ with  the shift-resistant graph $\mathds{H}\cap \binom{[n]\setminus \{a\}}{2}$.
\end{lem}

\begin{proof}
 Clearly $\ha$, $\hb$ are cross-intersecting and $\hb$ is intersecting. Since $\hf$ is of maximum size, $\ha$, $\hb$ are saturated cross-intersecting with respect to $\ha$. Moreover, $\tau(\hf)\geq 3$ implies the non-triviality of $\hb$, i.e., $\tau(\hb)\geq 2$.

 Fix an arbitrary $(i,j)\in \binom{[n]\setminus \{a\}}{2}$. If $(i,j)\notin \mathds{H}$, then $S_{ij}(\hf)=\hf$ and thereby $S_{ij}(\ha)=\ha$ and $S_{ij}(\hb)=\hb$. If $(i,j)\in \mathds{H}$, then $\tau(S_{ij}(\hf))=2$. Since every $T\in \hht^{(3)}(\hf)$  contains $a$, it follows that $\{a,i\}$ is a transversal of $S_{ij}(\hf)$. Then $S_{ij}(\hb)$ is a star. Hence, the shift-resistant graph for the shifted ad extremis cross-intersecting families $\ha$, $\hb$ is $\mathds{H}\cap \binom{[n]\setminus \{a\}}{2}$.
\end{proof}

Let $\ha\subset \binom{[n]}{\ell}$, $\hb\subset \binom{[n]}{k}$ be saturated cross-intersecting families with respect to $\ha$,  $n>\ell +k$ and let $\hb$ be non-trivial.  Suppose that $\ha,\hb$ are shifted ad extremis with respect to $\tau(\hb)\geq 2$. That is,  for each $(i,j)$, $1\leq i<j\leq n$, either (a) or (b) occurs.

\begin{itemize}
  \item[(a)] $S_{ij}(\ha)=\ha$, $S_{ij}(\hb)=\hb$;
  \item[(b)] $S_{ij}(\hb)$ is a star.
\end{itemize}

Let us define the {\it 2-cover graph} $\hhat$ for  $\hb$ on the vertex set $[n]$ where for $1\leq i<j\leq n$
\[
(i,j)\in \hhat \mbox{ \rm iff } \hb(\bar{i},\bar{j})=\emptyset.
\]

If for $D\subset [n]$, $\hf(D)=\binom{[n]\setminus D}{k-|D|}$ then we say that $D$ is {\it full} in $\hf$ or $\hf(D)$ is {\it full}. Since $\ha,\hb$ are saturated cross-intersecting with respect to $\ha$, we see that $\hhat$ is  the graph formed by all full pairs in $\ha$. Note that if $(i,j)$ is of type (b) then $(i,j)\in \hhat$ but not necessarily vice versa. Let $\mathds{H}$ be the corresponding shift-resistant graph. It is easy to see that  $\mathds{H}$ is a subgraph of $\hhat$.

For convenience, we often use $\{i,j\}\in \mathds{H}$ to denote that $(i,j)\in \mathds{H}$ for $i<j$ and $(j,i)\in \mathds{H}$ for $i>j$.  It is easy to checked that $\hg(n,k)$ is shifted ad extremis for $\tau(\hg(n,k))\geq 3$ with the shift-resistant graph
\[
\mathds{H}_0=\left\{\{i,j\}\colon i\in [2,k+1],\ j\in [k+2,2k]\right\}\cup \{(1,2),(1,3)\}.
\]
Let $\ha=\hg(n,k)(1)$ and  $\hb=\hg(n,k)(\bar{1})$. Clearly, $\ha,\hb$ are shifted ad extremis for $\tau(\hb)\geq 2$ with the shift-resistant graph
\[
\mathds{H}_1= \left\{\{i,j\}\colon i\in [2,k+1],\ j\in [k+2,2k]\right\}.
\]
and the 2-cover graph $\hhat_1=\mathds{H}_1 \cup \{(2,3)\}$.

\begin{defn}
We say that $\hhat$ is {\it partially shifted} if $(i,j)\in \hhat$ and $(x,j)\notin \hhat$ imply $S_{xj}(i,j)\in \hhat$. I.e., if $i<x$ then $(i,x)\in \hhat$; if $i>x$ then $(x,i)\in \hhat$.
\end{defn}

The reason to consider  $\hhat$ instead of $\mathds{H}$ is that $\hhat$ is partially shifted.

\begin{prop}\label{prop1.3}
Suppose that $\ha\subset \binom{[n]}{\ell}$, $\hb\subset \binom{[n]}{k}$ are cross-intersecting that were shifted ad extremis with respect to $\tau(\hb)\geq 2$, saturated with respect to $\ha$, but not both initial. If $n\geq \ell+k$, then $\hhat$ is partially shifted.
\end{prop}
\begin{proof}
Suppose for contradiction that $(x,y)\in \hhat$, $(z,y)\notin \hhat$ but $S_{zy}(x,y)\notin \hhat$. By the definition, we have $\hb(\bar{x},\bar{y})=\emptyset$.  Since $\ha,\hb$ are saturated cross-intersecting with respect to $\ha$, it follows that $\ha(x,y)$ is full, that is, $\ha(x,y)=\binom{[n]\setminus\{x,y\}}{\ell-2}$. Since $(z,y)\notin \hhat$, by the definition $S_{zy}(\ha)=\ha$, it follows that $\ha(x,z)$ is also full. Then by $\{x,z\}\notin \hhat$ we infer $\hb(\bar{x},\bar{z})\neq \emptyset$, contradicting the cross-intersecting property. Thus the proposition is proven.
\end{proof}

\begin{lem}\label{lem1.2}
If $\hhat$ is partially shifted and triangle-free, then $\hhat$ is a complete bipartite graph on partite sets $X$ and $Y$, $X\cup Y=[|X|+|Y|]$.
\end{lem}

\begin{proof}
Let us pick $(i,j)\in \hhat$ with $i+j$ minimal.

\begin{claim}
$i=1$.
\end{claim}
\begin{proof}
Indeed, if $i\geq 2$ then both $(1,i),(1,j)\notin \hhat$. By partial-shiftedness, $S_{1i}(i,j)=(1,j)\in \hhat$ and $S_{1j}(i,j)=(1,i)\in \hhat$, a contradiction.
\end{proof}

Let $Y=(y_1,\ldots,y_p)$ ($y_1<y_2<\ldots<y_p$) be the neighbors of $1$ in $\hhat$, $y_1=j$ and let $X=(x_1,\ldots,x_q)$ be the neighbors of $j=y_1$, $x_1=1$.

\begin{claim}
$\{x_u,y_v\}\in \hhat$ for all $1\leq u\leq q$, $1\leq v\leq p$.
\end{claim}
\begin{proof}
By symmetry assume $x_u<y_v$. If $x_u=1$, the statement holds by the definition. Assume $1<x_u$. If $(x_u,y_v)\notin \hhat$, then by partial-shiftedness $S_{x_uy_v}(1,y_v)=(1,x_u)\in \hhat$, implying that $(1,x_u,y_1)$ is a triangle of $\hhat$, contradiction.
\end{proof}

Since $\hhat$ is triangle-free, $X\cap Y=\emptyset$ and we proved that the graph $\hhat$ is complete on $X\times Y$. Let us show that $X\cup Y=[p+q]$. Indeed, otherwise we can find $z,w$ with $1<z<w$ and $z\notin X\cup Y$, $w\in X\cup Y$. Suppose by symmetry $w\in Y$. Then $(1,z)\notin \hhat$, $(y_1,z)\notin \hhat$. Note also that $\{y_1,w\}\notin \hhat$. If $y_1>w$ then $(1,y_1)\in \hhat$ implies $S_{w y_1}(1,y_1)=(1,w)\in \hhat$, a contradiction. Thus we may assume $y_1<w$.
 Should $(z,w)\in \hhat$ hold we infer $S_{y_1w}(z,w)=(y_1,z)\in \hhat$, a contradiction. Thus $(z,w)\notin \hhat$ and thereby $S_{zw}(1,w)=(1,z)\in \hhat$, a contradiction.

Let us show next that there are no edges $(u,v)\in \hhat$ outside $X\times Y$. Since any pair $(u,v)\subset Z$ for $Z=X$ or $Y$ would create a triangle, we may assume $v>p+q$. Now $1< y_1\leq p+q<v$ and the partial-shiftedness imply that $S_{1v}(u,v)=(1,u)\in \hhat$ and $S_{y_1v}(u,v)=\{y_1,u\}\in \hhat$. It follows that $(1,y_1,u)$ forms a triangle, contradiction. Thus $\hhat$ is a complete bipartite graph on partite sets $X$ and $Y$, $X\cup Y=[p+q]$.
\end{proof}

\begin{prop}\label{prop-5.4}
Suppose that $\ha\subset \binom{[n]}{\ell}$, $\hb\subset \binom{[n]}{k}$ are cross-intersecting that were shifted ad extremis with respect to $\tau(\hb)\geq 2$, saturated with respect to $\ha$, but not both initial. If  $\hhat$ is triangle-free, then  $\hhat$ is a complete bipartite graph on partite sets $X$ and $Y$, $X\cup Y=[|X|+|Y|]$, $2\leq |X|\leq k$, $2\leq |Y|\leq k$.
\end{prop}

\begin{proof}
By Lemma \ref{lem1.2} and Proposition \ref{prop1.3}, we are left to show that $2\leq |X|\leq k$, $2\leq |Y|\leq k$.

Let us prove $\min\{|X|,|Y|\}\geq 2$ first. The opposite would mean that $\hhat$ is a star. Thus it is sufficient to prove that $\hhat$ contains two independent edges. To this end fix an edge $(i,j)$ of type (b). Here is the place that we use that $\ha,\hb$ are not both initial.

Since $S_{ij}(\hb)$ is a star but $\hb$ is not, we may choose $K,L\in \hb$ with $K\cap (i,j)=\{i\}$,  $L\cap (i,j)=\{j\}$. Fix such a pair $K,L$ with $|K\cap L|$ maximal. Note that $\hb(i)\cap \hb(j)=\emptyset$ implies $|K\cap L|\leq k-2$. Pick $x\neq i$, $x\in K\setminus L$, $y\neq j$, $y\in L\setminus K$. We claim that $\{x,y\}$ is not of type (a). Indeed, if $x<y$ and $(x,y)$ is of type (a), then $S_{xy}(\hb)=\hb$ implies $(L\setminus\{y\})\cup\{x\}=:L'\in \hb$. But $|K\cap L'|=|K\cap L|+1$, a contradiction.

Similarly, $y<x$ and $S_{yx}(\hb)=\hb$ would imply $(K\setminus\{x\})\cup \{y\}=K'\in \hb$. Again, $|K'\cap L|=|K\cap L|+1$ provides the contradiction. Consequently, $\{x,y\}$ is of type (b), whence an edge of $\hhat$. Thus $\hhat$ contains two independent edges and therefore $\min\{|X|,|Y|\}\geq 2$.

Next we show that $\max\{|X|,|Y|\}\leq k$.


\begin{claim}\label{claim-6.7}
For every $B\in \hb$ either $X\subset B$ or $Y\subset B$.
\end{claim}
\begin{proof}
In the opposite case we may choose $x\in X\setminus B$, $y\in Y\setminus B$. Since $\{x,y\}\in\hhat$, $\ha(x,y)$ is full. Now $B\cap \{x,y\}=\emptyset$ contradicts the cross-intersecting property.
\end{proof}

Using the non-triviality of $\hb$, there are members $B, B'\in \hb$ with $X\subset B$, $Y\subset B'$. In particular, $|X|\leq k$ and $|Y|\leq k$. This concludes the proof of the proposition.
\end{proof}


\section{The case that $\hht^{(3)}(\hf)$ is non-trivial}

In this section, we prove an important special case of our main result.

\begin{thm}\label{thm-case1}
Let $\hf\subset \binom{[n]}{k}$ be an intersecting family with $n> 2k$, $\tau(\hf)\geq 3$ and $|\hf|$ maximal. Suppose that $\hf$ is shifted ad extremis with a non-empty shift-resistant graph $\mathds{H}$, and $\hht^{(3)}(\hf)$ is non-trivial. Then $|\hf|<|\hg(n,k)|$ for $k\geq 7$.
\end{thm}

For the proof of Theorem \ref{thm-case1}, we need the following computational lower bounds for $|\hg(n,k)|$.

\begin{lem}
For $n\geq 4k$ and $k\geq 7$,
\begin{align}
&|\hg(n,k)|>3\left(\binom{n-6}{k-2}-\binom{n-k-6}{k-2}\right)+\binom{n-3}{k-3}+3\binom{n-4}{k-3}
 +6\binom{n-5}{k-3},\label{ineq-key16} \\[3pt]
&|\hg(n,k)|>4\left(\binom{n-5}{k-2}-\binom{n-k-5}{k-2}\right)+\binom{n-3}{k-3}+3\binom{n-4}{k-3}
 +4\binom{n-5}{k-3}.\label{ineq-key17}
\end{align}
\end{lem}

\begin{proof}
For $k\geq 7$, we have
\begin{align}
|\hg(n,k)|> \binom{n-1}{k-1}-\binom{n-k}{k-1}-\binom{n-k-1}{k-1}+\binom{n-2k}{k-1}&= \sum_{i=2}^{k} \left(\binom{n-i}{k-2} -\binom{n-k-i}{k-2}\right)\nonumber \\[3pt]
&\geq \sum_{i=2}^{7} \left(\binom{n-i}{k-2} -\binom{n-k-i}{k-2}\right).\label{ineq-new5.25}
\end{align}

First we prove \eqref{ineq-key16}. The RHS of \eqref{ineq-key16} is equal to
 \begin{align*}
\sum_{2\leq i\leq 4}&\left(\binom{n-i}{k-2}-\binom{n-k-i}{k-2}\right)+\binom{n-4}{k-3}
 +3\binom{n-5}{k-3}-3\binom{n-6}{k-3}\\[3pt]
 &\qquad\qquad +\binom{n-k-3}{k-3}+2\binom{n-k-4}{k-3}+3\binom{n-k-5}{k-3}+3\binom{n-k-6}{k-3}.
 \end{align*}
 By $k\geq 7$, it is at most
 \begin{align*}
  &\quad \sum_{2\leq i\leq 4}\left(\binom{n-i}{k-2}-\binom{n-k-i}{k-2}\right)+\binom{n-4}{k-3}
 +3\binom{n-5}{k-3}-3\binom{n-6}{k-3}+\binom{n-10}{k-3}\\[3pt]
 &\qquad\qquad+2\binom{n-11}{k-3}+3\binom{n-12}{k-3}+3\binom{n-13}{k-3}\\[3pt]
 &\leq \sum_{2\leq i\leq 4}\left(\binom{n-i}{k-2}-\binom{n-k-i}{k-2}\right)+\binom{n-4}{k-3}
 +3\binom{n-5}{k-3}+3\binom{n-12}{k-3}+3\binom{n-13}{k-3}.
 \end{align*}
  By \eqref{ineq-new5.25}, it suffices to show that
  \begin{align*}
  \binom{n-4}{k-3}
 +3\binom{n-5}{k-3}+3\binom{n-12}{k-3}+3\binom{n-13}{k-3}\leq \sum_{5\leq i\leq 7}\left(\binom{n-i}{k-2}-\binom{n-k-i}{k-2}\right).
  \end{align*}
For $n\geq 4k$, by \eqref{ineq-key13} we have
\begin{align*}
&\binom{n-6}{k-3}+3\binom{n-9}{k-3}\geq \left(\left(\frac{3}{4}\right)^2+3\left(\frac{3}{4}\right)^5\right)\binom{n-4}{k-3}> \binom{n-4}{k-3}
\end{align*}
and
\begin{align*}
2\binom{n-7}{k-3}+3\binom{n-8}{k-3}+3\binom{n-10}{k-3}& \geq \left(2\left(\frac{3}{4}\right)^2+3\left(\frac{3}{4}\right)^3
+3\left(\frac{3}{4}\right)^5\right)\binom{n-5}{k-3}> 3\binom{n-5}{k-3}.
\end{align*}
Adding the above two inequalities, we get
\[
\binom{n-4}{k-3}+3\binom{n-5}{k-3}<\binom{n-6}{k-3}+2\binom{n-7}{k-3}+3\binom{n-8}{k-3}
+3\binom{n-9}{k-3}+3\binom{n-10}{k-3}.
\]
Therefore,
\begin{align*}
 &\quad\ \binom{n-4}{k-3}
 +3\binom{n-5}{k-3}+3\binom{n-12}{k-3}+3\binom{n-13}{k-3}\\[3pt]
 &\leq \binom{n-6}{k-3}+2\binom{n-7}{k-3}
+3\binom{n-8}{k-3}+3\binom{n-9}{k-3}+3\binom{n-10}{k-3}+3\binom{n-12}{k-3}
+3\binom{n-13}{k-3}\\[3pt]
 &\leq \binom{n-6}{k-3}+2\binom{n-7}{k-3}
+3\binom{n-8}{k-3}+3\binom{n-9}{k-3}+3\binom{n-10}{k-3}+3\binom{n-12}{k-3}
+3\binom{n-11}{k-3}\\[3pt]
&\leq \sum_{5\leq i\leq 7}\sum_{1\leq j\leq 7} \binom{n-i-j}{k-3}\leq \sum_{5\leq i\leq 7}\left(\binom{n-i}{k-2}-\binom{n-k-i}{k-2}\right)
\end{align*}
and \eqref{ineq-key16} follows.

For \eqref{ineq-key17}, by $k\geq 7$ the RHS of \eqref{ineq-key17} is at most
\begin{align*}
&\sum_{2\leq i\leq 5}\left(\binom{n-i}{k-2}-\binom{n-k-i}{k-2}\right)+\binom{n-4}{k-3}
 +\binom{n-5}{k-3}+\binom{n-k-3}{k-3}+2\binom{n-k-4}{k-3}\\[3pt]
 &\qquad\qquad+3\binom{n-k-5}{k-3}\\[3pt]
\leq &\sum_{2\leq i\leq 5}\left(\binom{n-i}{k-2}-\binom{n-k-i}{k-2}\right)+\binom{n-4}{k-3}
 +\binom{n-5}{k-3}+\binom{n-10}{k-3}+2\binom{n-11}{k-3}+3\binom{n-12}{k-3}.
 \end{align*}
 Then it suffices to prove that
 \[
 \binom{n-4}{k-3}
 +\binom{n-5}{k-3}+\binom{n-10}{k-3}+2\binom{n-11}{k-3}+3\binom{n-12}{k-3}\leq \sum_{6\leq i\leq 7}\left(\binom{n-i}{k-2}-\binom{n-k-i}{k-2}\right).
 \]
 For $n\geq 4k$, by \eqref{ineq-key13} we infer that
\begin{align*}
&\frac{4}{7}\binom{n-7}{k-3}+\binom{n-8}{k-3}+\binom{n-9}{k-3}+\binom{n-10}{k-3}
+\frac{4}{9}\binom{n-13}{k-3}\\[3pt]
&\qquad\qquad\geq \left(\frac{4}{7}\times\left(\frac{3}{4}\right)^3+\left(\frac{3}{4}\right)^4+\left(\frac{3}{4}\right)^5
+\left(\frac{3}{4}\right)^6+\frac{4}{9}\times\left(\frac{3}{4}\right)^9\right)\binom{n-4}{k-3}
>\binom{n-4}{k-3},\\[3pt]
&\frac{3}{7}\binom{n-7}{k-3}+\binom{n-8}{k-3}+\binom{n-9}{k-3}
+\frac{2}{9}\binom{n-13}{k-3}\\[3pt]
&\qquad\qquad\geq \left(\frac{3}{7}\times\left(\frac{3}{4}\right)^2+\left(\frac{3}{4}\right)^3+\left(\frac{3}{4}\right)^4
+\frac{2}{9}\times\left(\frac{3}{4}\right)^8\right)\binom{n-5}{k-3}
>\binom{n-5}{k-3},\\[3pt]
&\frac{4}{3}\binom{n-13}{k-3}\geq
\frac{4}{3}\times\frac{3}{4}\binom{n-12}{k-3}=\binom{n-12}{k-3}.
\end{align*}
Adding the above three inequalities, we get
\begin{align*}
\binom{n-4}{k-3}&+\binom{n-5}{k-3}+\binom{n-12}{k-3}\\[3pt]
&<\binom{n-7}{k-3}
+2\binom{n-8}{k-3}
+2\binom{n-9}{k-3}+\binom{n-10}{k-3}+2\binom{n-13}{k-3}.
\end{align*}
Thus,
 \begin{align*}
 &\,\binom{n-4}{k-3}
 +\binom{n-5}{k-3}+\binom{n-10}{k-3}+2\binom{n-11}{k-3}+3\binom{n-12}{k-3}\\[3pt]
 \leq&\,\binom{n-7}{k-3}+2\binom{n-8}{k-3}+2\binom{n-9}{k-3}
+2\binom{n-10}{k-3}+2\binom{n-11}{k-3}+2\binom{n-12}{k-3}+2\binom{n-13}{k-3}\\[3pt]
 \leq&\,\sum_{6\leq i\leq 7}\sum_{1\leq j\leq 7} \binom{n-i-j}{k-3}\leq \sum_{6\leq i\leq 7}\left(\binom{n-i}{k-2}-\binom{n-k-i}{k-2}\right)
 \end{align*}
 and \eqref{ineq-key17} follows.
\end{proof}

\begin{lem}
For $n\geq 2k$ and $k\geq 7$,
\begin{align}
&\prod_{2\leq i\leq k-1}\frac{t+k-6+i}{t+i} >\left(\frac{2t+3k-11}{2t+k+1}\right)^{k-2},\label{ineq-key10-4}\\[3pt]
&\prod_{2\leq i\leq k-1}\frac{t+k-5+i}{t+i}
>\left(\frac{2t+3k-9}{2t+k+1}\right)^{k-2}.\label{ineq-key10-42}
\end{align}
\end{lem}

\begin{proof}
Note that for $m>d>i>0$,
\begin{align}\label{ineq-key10-5}
\frac{m-d-i}{m-i}\cdot \frac{m-d+i}{m+i} <\left(\frac{m-d}{m}\right)^2.
\end{align}
Equivalently,
\begin{align*}
&\frac{(m-d)^2-i^2}{(m-d)^2}<\frac{m^2-i^2}{m^2}, \mbox{ that is, }\left(\frac{i}{m}\right)^2<\left(\frac{i}{m-d}\right)^2,
\end{align*}
which is true for $m>d>0$.
Applying \eqref{ineq-key10-5} repeatedly with $m=t+\frac{3k}{2}-\frac{11}{2}$ and $d=k-6$, we obtain
\[
\frac{(t+2)(t+3)\ldots (t+k-1)}{(t+k-4)(t+k-3)\ldots (t+2k-7)} <\left(\frac{t+\frac{k}{2}+\frac{1}{2}}{t+\frac{3k}{2}-\frac{11}{2}}\right)^{k-2}
=\left(\frac{2t+k+1}{2t+3k-11}\right)^{k-2}.
\]
and \eqref{ineq-key10-4} follows. Similarly, we can obtain \eqref{ineq-key10-42}.
\end{proof}

\begin{lem}
For $2k+1\leq n\leq 4k$ and $k\geq 7$,
\begin{align}
&\binom{n-6}{k-1}+2\binom{n-6}{k-2}-2\binom{n-6}{k-4}-\binom{n-6}{k-5}>\binom{n-k}{k-1}
+\binom{n-k-1}{k-1},\label{ineq-key14}\\[3pt]
&\binom{n-5}{k-1}-\binom{n-5}{k-4}>\binom{n-k}{k-1}+\binom{n-k-1}{k-1}.\label{ineq-key15}
\end{align}
\end{lem}

\begin{proof}
Let us prove \eqref{ineq-key14} first. Set $t=n-2k$. Then $1\leq t\leq 2k$. Note that
\begin{align*}
\binom{n-6}{k-2}& = \frac{k-1}{n-k-4}\binom{n-6}{k-1}=\frac{k-1}{t+k-4}\binom{n-6}{k-1},\\[3pt]
\binom{n-6}{k-4} &= \frac{(k-1)(k-2)(k-3)}{(n-k-2)(n-k-3)(n-k-4)}\binom{n-6}{k-1}\\[3pt]
&=\frac{(k-1)(k-2)(k-3)}{(t+k-2)(t+k-3)(t+k-4)}\binom{n-6}{k-1},\\[3pt]
\binom{n-6}{k-5} &= \frac{(k-1)(k-2)(k-3)(k-4)}{(n-k-1)(n-k-2)(n-k-3)(n-k-4)}\binom{n-6}{k-1}\\[3pt]
&=\frac{(k-1)(k-2)(k-3)(k-4)}{(t+k-1)(t+k-2)(t+k-3)(t+k-4)}\binom{n-6}{k-1}.
\end{align*}
Moreover,
\[
\binom{n-k-1}{k-1}= \frac{n-2k+1}{n-k}\binom{n-k}{k-1} = \frac{t+1}{t+k}\binom{n-k}{k-1}
\]
and
\[
\frac{\binom{n-6}{k-1}}{\binom{n-k}{k-1}} =\prod_{0\leq i\leq k-2}\frac{n-6-i}{n-k-i}=\prod_{2\leq i\leq k}\frac{t+k-6+i}{t+i}.
\]
Since
\begin{align*}
&\qquad\frac{2(k-1)(k-2)(k-3)}{(t+k-2)(t+k-3)(t+k-4)}+
\frac{(k-1)(k-2)(k-3)(k-4)}{(t+k-1)(t+k-2)(t+k-3)(t+k-4)}\\[3pt]
&=\frac{(k-1)(k-2)(k-3)}{(t+k-2)(t+k-3)(t+k-4)}\left(2+\frac{k-4}{t+k-1}\right)\\[3pt]
&=\frac{(k-1)(k-2)(k-3)(2t+3k-6)}{(t+k-1)(t+k-2)(t+k-3)(t+k-4)},
\end{align*}
we  see that \eqref{ineq-key14} is equivalent to
\begin{align*}
&\left(1+\frac{2(k-1)}{t+k-4}-\frac{(k-1)(k-2)(k-3)(2t+3k-6)}{(t+k-1)(t+k-2)(t+k-3)(t+k-4)}\right)\prod_{2\leq i\leq k}\frac{t+k-6+i}{t+i}> \frac{2t +k+1}{t+k}.
\end{align*}
By moving out the first term $\frac{t+k-4}{t+2}$ and the last term $\frac{t+2k-6}{t+k}$ from  $\prod$, we get
\begin{align*}
&\left(t+3k-6-\frac{(k-1)(k-2)(k-3)(2t+3k-6)}{(t+k-1)(t+k-2)(t+k-3)}\right)\prod_{3\leq i\leq k-1}\frac{t+k-6+i}{t+i}> \frac{(t+2)(2t +k+1)}{t+2k-6}.
\end{align*}
By further moving out the terms $\frac{t+k-3}{t+3}\frac{t+k-2}{t+4}\frac{t+k-1}{t+5}$ from  $\prod$,
\begin{align*}
&\left((t+3k-6)(t+k-1)(t+k-2)(t+k-3)-(k-1)(k-2)(k-3)(2t+3k-6)\right)\\[3pt]
&\cdot\prod_{6\leq i\leq k-1}\frac{t+k-6+i}{t+i}> \frac{(t+2)(t+3)(t+4)(t+5)(2t +k+1)}{t+2k-6}.
\end{align*}
By factorization, we have
\begin{align*}
&(t+3k-6)(t+k-1)(t+k-2)(t+k-3)-(k-1)(k-2)(k-3)(2t+3k-6)\\[3pt]
&=t(t+2k-5)(t+2k-4)(t+2k-3).
\end{align*}
It follows that
\begin{align*}
&\prod_{6\leq i\leq k-1}\frac{t+k-6+i}{t+i}> \frac{(t+2)(t+3)(t+4)(t+5)(2t +k+1)}{t(t+2k-6)(t+2k-5)(t+2k-4)(t+2k-3)}.
\end{align*}
Equivalently,
\begin{align}\label{ineq-10-1}
&\prod_{2\leq i\leq k-1}\frac{t+k-6+i}{t+i}> \frac{(t+k-4)(t+k-3)(t+k-2)(t+k-1)(2t +k+1)}{t(t+2k-6)(t+2k-5)(t+2k-4)(t+2k-3)}.
\end{align}
For $k=7,8,9,10,11$ and $t\leq 2k$, it can be checked  directly that \eqref{ineq-10-1} holds. Thus we may assume $k\geq 12$.

Since $\frac{2t+k+1}{t}=2+\frac{k+1}{t}\leq k+3$ and by $k\geq 12$
\begin{align*}
\frac{(t+k-4)(t+k-3)(t+k-2)(t+k-1)}{(t+2k-6)(t+2k-5)(t+2k-4)(t+2k-3)}&\leq \frac{(3k-4)(3k-3)(3k-2)(3k-1)}{(4k-6)(4k-5)(4k-4)(4k-3)}\\[3pt]
&\leq \frac{32\times 33\times 34\times 35}{42\times 43\times 44\times 45}\approx 0.3514<\frac{2}{5},
\end{align*}
we see that the RHS of \eqref{ineq-10-1} is less than $\frac{2(k+3)}{5}$. By \eqref{ineq-key10-4}, $t\leq 2k$ and  $k\geq 12$, we infer that the LHS of \eqref{ineq-10-1} is greater than
\[
\left(\frac{2t+3k-11}{2t+k+1}\right)^{k-2}\geq \left(\frac{7k-11}{5k+1}\right)^{k-2}\geq \left(\frac{73}{61}\right)^{k-2}.
\]
We prove $\left(\frac{73}{61}\right)^{k-2}> \frac{2(k+3)}{5}$ for $k\geq 12$ by induction.
For $k=12$ we have $\left(\frac{73}{61}\right)^{10}\approx 6.0246 >6=\frac{2\times(12+3)}{5}$. For $k+1\geq 13$, by induction hypothesis
\[
\left(\frac{73}{61}\right)^{k-1} = \left(\frac{73}{61}\right)^{k-2}+\frac{12}{61}\left(\frac{73}{61}\right)^{k-2}>\frac{2(k+3)}{5}+ \frac{12}{61} \left(\frac{73}{61}\right)^{k-2}>\frac{2(k+4)}{5}.
\]

Now  let us prove \eqref{ineq-key15}. Similarly, set $t=n-2k$ and then $1\leq t\leq 2k$. Note that
\begin{align*}
\binom{n-5}{k-4} &=\frac{(k-1)(k-2)(k-3)}{(n-k-1)(n-k-2)(n-k-3)}\binom{n-5}{k-1}=\frac{(k-1)(k-2)(k-3)}{(t+k-1)(t+k-2)(t+k-3)}\binom{n-5}{k-1}.
\end{align*}
Then \eqref{ineq-key15} is equivalent to
\[
\left(1-\frac{(k-1)(k-2)(k-3)}{(t+k-1)(t+k-2)(t+k-3)}\right)\prod_{2\leq i\leq k}\frac{t+k-5+i}{t+i}> \frac{2t +k+1}{t+k}.
\]
By moving the last term $\frac{t+2k-5}{t+k}$ from $\prod$ to the RHS, we obtain
\begin{align}\label{ineq-10-2}
\left(1-\frac{(k-1)(k-2)(k-3)}{(t+k-1)(t+k-2)(t+k-3)}\right)\prod_{2\leq i\leq k-1}\frac{t+k-5+i}{t+i}> \frac{2t +k+1}{t+2k-5}.
\end{align}
For $k=7,8,9,10$ and $t\leq 2k$, it can checked directly that \eqref{ineq-10-2} holds. Thus we may assume $k\geq 11$.

Note that
\[
1-\frac{(k-1)(k-2)(k-3)}{(t+k-1)(t+k-2)(t+k-3)} \geq 1-\frac{k-3}{k}=\frac{3}{k}
\]
and by  \eqref{ineq-key10-42} and $k\geq 11$
\[
\prod_{2\leq i\leq k-1}\frac{t+k-5+i}{t+i} \geq \left(\frac{2t+3k-9}{2t+k+1}\right)^{k-2}\geq \left(\frac{7k-9}{5k+1}\right)^{k-2} \geq \left(\frac{17}{14}\right)^{k-2}.
\]
Moreover, $t\leq 2k$ and $k\geq 11$ imply
\[
\frac{2t +k+1}{t+2k-5}= 2-\frac{3k-11}{t+2k-5} \leq 2-\frac{3k-11}{4k-5}\leq 2-\frac{22}{39}<\frac{3}{2}.
\]
Thus, for $k\geq 11$ it suffices to show $\left(\frac{17}{14}\right)^{k-2} > \frac{k}{2}$. We prove it by induction on $k$.
For $k=11$, we have $\left(\frac{17}{14}\right)^{9}\approx5.7397 > \frac{11}{2}$. For $k+1\geq 12$, by induction hypothesis we have
\[
\left(\frac{17}{14}\right)^{k-1} >\left(\frac{17}{14}\right)^{k-2}+\frac{3}{14}\left(\frac{17}{14}\right)^{k-2}> \frac{k}{2}+\frac{3}{14}\left(\frac{17}{14}\right)^{k-2}>\frac{k+1}{2}.
\]
\end{proof}

Now we are ready to prove  Theorem \ref{thm-case1}.

\begin{proof}[Proof of Theorem \ref{thm-case1}]
By the maximality of $|\hf|$, we infer that $T$ is full in $\hf$ for all $T\in \hht^{(3)}(\hf)$. It follows that $\hht^{(3)}(\hf)$ is intersecting.

\begin{claim}\label{claim-6.4}
There exists $\{a,b,c\},\{d,e,f\}\in \hht^{(3)}(\hf)$ such that $(a,b),(d,e)\in\mathds{H}$, $\{a,c\}$ is a transversal of $S_{ab}(\hf)$, $\{d,f\}$ is a transversal of $S_{de}(\hf)$, $\{a,b,c\}\cap (d,e)=\emptyset$ and $f\in \{a,b,c\}$.
\end{claim}

\begin{proof}
Since $\mathds{H}\neq \emptyset$, let $(a,b)\in \mathds{H}$. Then $\tau(S_{ab}(\hf))\leq 2$. Using $\tau(\hf)\geq 3$ we must have equality. Fix $c$ so that $(a,c)$ is a transversal of $S_{ab}(\hf)$ and note that $\{a,b,c\}$ is a transversal of $\hf$. Hence the maximality of $|\hf|$ implies that $\{a,b,c\}$ is full, i.e., $\{a,b,c\}\subset F\in \binom{[n]}{k}$ implies $F\in \hf$.

Since $\tau(\hf)>2$, $(a,b), (a,c),(b,c)$ are not transversals of $\hf$. Thus for each $x\in \{a,b,c\}$ we can fix $F_x\in \hf$ with $F_x\cap \{a,b,c\} =\{x\}$.

We claim that there is some $(d,e)\in \mathds{H}$ with $\{a,b,c\}\cap (d,e)=\emptyset$. In the opposite case $\hf$ is initial on $[n]\setminus \{a,b,c\}$. Let $E\in \binom{[n]\setminus \{a,b,c\}}{k-1}$ consist of the smallest $(k-1)$ elements of $[n]\setminus \{a,b,c\}$. Then $E\cup\{x\}\prec F_x$ implies $E\cup \{x\}\in \hf$. Now $E\in \hf(a)$, $E\in \hf(b)$ imply $E\cup \{b\}\in S_{ab}(\hf)$. However $(E\cup \{b\})\cap \{a,c\}=\emptyset$, i.e., $\{a,c\}$ is not a transversal of $S_{ab}(\hf)$, a contradiction.

Now $(d,e)\in \mathds{H}$ implies $\tau(S_{de}(\hf))= 2$. Let $\{d,f\}$ be a transversal of $S_{de}(\hf)$. Since $\hht^{(3)}(\hf)$ is intersecting, $\{a,b,c\}\cap (d,e)=\emptyset$ implies $f\in \{a,b,c\}$. Thus the claim holds.
\end{proof}

\begin{claim}\label{claim-6.8}
Suppose that $(x,y)\in \mathds{H}$ and $(x,z)\in \hht^{(2)}(S_{xy}(\hf))$.
If $E\in \hf(x,\bar{y})\cap \hf(\bar{x},y)$ then $z\in E$.
\end{claim}
\begin{proof}
Suppose that $z\notin E$. Then $E\in \hf(x,\bar{y})\cap \hf(\bar{x},y)$ implies $E\cup \{y\} \in S_{xy}(\hf)$. However, $(E\cup \{y\})\cap (x,z)=\emptyset$, contradicting the fact that $(x,z)\in \hht^{(2)}(S_{xy}(\hf))$.
\end{proof}

\begin{fact}
Let $\hf\subset \binom{[n]}{k}$ be an intersecting family with $\tau(\hf)\geq 3$. Then for every $P\in \binom{[n]}{2}$ and $Q\subset [n]\setminus P$,
\begin{align}\label{ineq-key3}
|\hf(P,P\cup Q)| \leq \binom{n-|Q|-2}{k-2}-\binom{n-k-|Q|-2}{k-2}.
\end{align}
\end{fact}
\begin{proof}

Since $\tau(\hf)\geq 3$, there exists $F(P)\in \hf$ such that $P\cap F(P)=\emptyset$. By the intersection property of $\hf$, we infer
\begin{align*}
|\hf(P,P\cup Q)| &\leq \binom{n-|Q|-2}{k-2}- \binom{n-|Q|-2-|F(P)\setminus Q|}{k-2}\\[3pt]
&\leq \binom{n-|Q|-2}{k-2}- \binom{n-k-|Q|-2}{k-2}.
\end{align*}
\end{proof}

We first prove the theorem for $n\geq 4k$.

\begin{prop}\label{prop-6.5}
If $\hht^{(3)}(\hf)$ is non-trivial, then $|\hf|<|\hg(n,k)|$ for $n\geq 4k$ and $k\geq 7$.
\end{prop}

\begin{proof}
By Claim \ref{claim-6.4}, there are $\{a,b,c\}$, $\{d,e,f\} \in \hht^{(3)}(\hf)$ such that $(a,b),(d,e)\in\mathds{H}$, $\{a,c\}$ is a transversal of $S_{ab}(\hf)$, $\{d,f\}$ is a transversal of $S_{de}(\hf)$, $\{a,b,c\}\cap (d,e)=\emptyset$ and $f\in \{a,b,c\}$. Let $(v,w)=\{a,b,c\}\setminus \{f\}$. Since $\hht^{(3)}(\hf)$ is non-trivial, there exists $T\in \hht^{(3)}(\hf)$ such that $f\notin T$. Clearly $T\cap (v,w)\neq\emptyset \neq T\cap (d,e)$.

There are essentially two possibilities for $T$.

 {\bf\noindent Case 1.} $T=\{w,d,g\}$ with $g\notin \{v,w,d,e,f\}$.

 Let $\hp=\{\{v,w,f\},\{d,e,f\},\{w,d,g\}\}\subset \hht^{(3)}(\hf)$ and $U= \{v,w,d,e,f,g\}$. Then for any $R\in \binom{[n]\setminus U}{k-2}$ and $S\in\binom{U}{2}$ with $R\cup S\in \hf$, we have $S\in \hht^{(2)}(\hp)$. It is easy to see that
 \[
 \hht^{(2)}(\hp)=\left\{\{v,d\},\{w,d\},\{w,e\},\{w,f\},\{d,f\},\{f,g\}\right\}.
 \]
 Note that the non-triviality of $\hp$ implies $|F\cap U|\geq 2$ for all $F\in \hf$.
 Define
 \[
 \hf_i=\left\{F\in \hf\colon |F\cap U|=i\right\},\ i=2,3,\ldots,6.
 \]

 \begin{claim}\label{claim-new6.7}
 \begin{align}\label{ineq-hf2}
 |\hf_2| \leq 3\left(\binom{n-6}{k-2}-\binom{n-k-6}{k-2}\right).
 \end{align}
 \end{claim}

 \begin{proof}
 Note that $\hf(\{w,d\},U)$ and  $\hf(\{f,g\},U)$ are cross-intersecting. If one of them is empty, then by \eqref{ineq-key3} we infer
 \[
 |\hf(\{w,d\},U)|+|\hf(\{f,g\},U)| \leq \binom{n-6}{k-2}-\binom{n-k-6}{k-2}.
 \]
 If both of them are non-empty, then by \eqref{ft92}
 \[
 |\hf(\{w,d\},U)|+|\hf(\{f,g\},U)| \leq \binom{n-6}{k-2}-\binom{n-k-4}{k-2}+1< \binom{n-6}{k-2}-\binom{n-k-6}{k-2}.
 \]
 Similarly, we have
 \begin{align*}
 |\hf(\{w,e\},U)|+|\hf(\{d,f\},U)| &\leq \binom{n-6}{k-2}-\binom{n-k-6}{k-2}\\[3pt]
  |\hf(\{w,f\},U)|+|\hf(\{v,d\},U)| &\leq \binom{n-6}{k-2}-\binom{n-k-6}{k-2}.
 \end{align*}
 Adding these inequalities, we obtain \eqref{ineq-hf2}.
 \end{proof}

Let $T\in \binom{U}{3}$. By \eqref{ineq-sperner} we have
\[
|\hf(T,U)| +|\hf(U\setminus T,U)| \leq \binom{n-6}{k-3}.
\]
 It follows that $|\hf_3| \leq \frac{1}{2}\binom{6}{3}\binom{n-6}{k-3}=10\binom{n-6}{k-3}$.
 Hence, by \eqref{ineq-key16} we obtain that
 \begin{align*}
 |\hf|  &= |\hf_2|+|\hf_3|+|\hf_4|+|\hf_5|+|\hf_6|\\[3pt]
 &\leq 3\left(\binom{n-6}{k-2}-\binom{n-k-6}{k-2}\right)+10\binom{n-6}{k-3}
 +15\binom{n-6}{k-4}+6\binom{n-6}{k-5}+\binom{n-6}{k-6}\\[3pt]
 &= 3\left(\binom{n-6}{k-2}-\binom{n-k-6}{k-2}\right)+\binom{n-3}{k-3}+3\binom{n-4}{k-3}
 +6\binom{n-5}{k-3}\\[3pt]
&\overset{\eqref{ineq-key16}}{<}|\hg(n,k)|.
 \end{align*}

 {\bf\noindent Case 2.} $T=\{w,d,e\}$.

 Let $\hp=\{\{v,w,f\},\{d,e,f\},\{w,d,e\}\}\subset \hht^{(3)}(\hf)$ and $U= \{v,w,d,e,f\}$. Then for any $R\in \binom{[n]\setminus U}{k-2}$ and $S\in\binom{U}{2}$ with $R\cup S\in \hf$, we have $S\in \hht^{(2)}(\hp)$. Note that $|U|=5=3+2$ implies $\hht^{(2)}(\hp)=\binom{U}{2}\setminus \{U\setminus P\colon P\in \hp\}$. That is,
 \[
 \hht^{(2)}(\hp)=\left\{\{v,d\},\{v,e\},\{w,d\},\{w,e\},\{w,f\},\{d,f\},\{e,f\}\right\}.
 \]
 Define $\hf_i=\{F\in \hf\colon |F\cap U|=i\}$, $i=2,3,\ldots,6$. Note that $(\hf(\{w,d\}, U), \hf(\{e,f\}, U))$,
 $(\hf(\{w,e\}, U), \hf(\{v,d\}, U))$, $(\hf(\{w,f\}, U), \hf(\{v,e\}, U))$ are three cross-intersecting pairs.  By a similar argument as in the proof of Claim \ref{claim-new6.7}, we obtain that
  \begin{align*}
 |\hf(\{w,d\},U)|+|\hf(\{e,f\},U)| &\leq \binom{n-5}{k-2}-\binom{n-k-5}{k-2},\\[3pt]
 |\hf(\{w,e\},U)|+|\hf(\{v,d\},U)| &\leq \binom{n-5}{k-2}-\binom{n-k-5}{k-2},\\[3pt]
 |\hf(\{w,f\},U)|+|\hf(\{v,e\},U)| &\leq \binom{n-5}{k-2}-\binom{n-k-5}{k-2}.
 \end{align*}
 Moreover, by \eqref{ineq-key3} we have
   \begin{align*}
 |\hf(\{d,f\},U)| &\leq \binom{n-5}{k-2}-\binom{n-k-5}{k-2}.
 \end{align*}
Thus,
 \begin{align}\label{ineq-hf22}
 |\hf_2| =\sum_{P\in \hht^{(2)}(\hp)} |\hf(P,U)| \leq 4\left(\binom{n-5}{k-2}-\binom{n-k-5}{k-2}\right).
 \end{align}
Recall that $\{a,c\}$ is a transversal of $S_{ab}(\hf)$ and $\{d,f\}$ is a transversal of $S_{de}(\hf)$.  We claim that for any $R\in \binom{[n]\setminus U}{k-3}$, at most one of $R\cup \{a,d,e\}$, $R\cup \{b,d,e\}$ is in $\hf$. Indeed, otherwise if $R\cup \{a,d,e\}, R\cup \{b,d,e\}\in \hf$ then $R\cup \{d,e\}\in \hf(a,\bar{b})\cap \hf(\bar{a},b)$ but $c\notin R\cup \{d,e\}$, contradicting Claim \ref{claim-6.8}. Thus for any $R\in \binom{[n]\setminus U}{k-3}$, at most one of $R\cup \{a,d,e\}$, $R\cup \{b,d,e\}$ is in $\hf$. Similarly, at most one of $R\cup \{d,v,w\}$, $R\cup \{e,v,w\}$ is in $\hf$.   It follows that
\begin{align*}
&|\hf(\{a,d,e\},U)|+|\hf(\{b,d,e\},U)|\leq \binom{n-5}{k-3},\\[3pt]
&|\hf(\{d,v,w\},U)|+|\hf(\{e,v,w\},U)|\leq \binom{n-5}{k-3}.
\end{align*}
Consequently,
\[
|\hf_3|=\sum_{B\in \binom{U}{3}} |\hf(B,U)|\leq 8\binom{n-5}{k-3}.
\]
Therefore, by \eqref{ineq-key17} we conclude that
 \begin{align*}
 |\hf| &= |\hf_2|+|\hf_3|+|\hf_4|+|\hf_5|\\[3pt]
 &\leq 4\left(\binom{n-5}{k-2}-\binom{n-k-5}{k-2}\right)+8\binom{n-5}{k-3}
 +5\binom{n-5}{k-4}+\binom{n-5}{k-5}\\[3pt]
 &= 4\left(\binom{n-5}{k-2}-\binom{n-k-5}{k-2}\right)+\binom{n-3}{k-3}+3\binom{n-4}{k-3}
 +4\binom{n-5}{k-3}\\[3pt]
 &\overset{\eqref{ineq-key17}}{<}|\hg(n,k)|.
 \end{align*}
\end{proof}

In Proposition \ref{prop-6.5}, we proved that if  $\hht^{(3)}(\hf)$ is non-trivial then $|\hf|< |\hg(n,k)|$ for $n\geq 4k$ and $k\geq 7$. Let us  prove that  $|\hf|< |\hg(n,k)|$ for $k\geq 7$ in the full range.

\begin{prop}\label{prop-6.9}
If $\hht^{(3)}(\hf)$ is non-trivial, then $|\hf|<|\hg(n,k)|$ for $n\geq 2k+1$ and $k\geq 7$.
\end{prop}

\begin{proof}
By Proposition \ref{prop-6.5}, we may assume that $2k<n< 4k$. By Claim \ref{claim-6.4}, let $\{a,b,c\}$, $\{d,e,f\}$ be the two transversals of $\hf$ with $\{a,b,c\}\cap \{d,e,f\}=\{f\}$ and let $(v,w)=\{a,b,c\}\setminus \{f\}$. Since $\hht^{(3)}(\hf)$ is non-trivial, there exists $T\in \hht^{(3)}(\hf)$ such that $f\notin T$. Clearly $T\cap (v,w)\neq\emptyset \neq T\cap (d,e)$.

 We distinguish two cases as above. However the computation is very much different for $2k<n<4k$.

 {\bf\noindent Case 1.} $T=\{w,d,g\}$ with $g\notin \{v,w,d,e,f\}$.

 Let $\hp=\{\{v,w,f\},\{d,e,f\},\{w,d,g\}\}\subset \hht^{(3)}(\hf)$ and $D= \{v,w,d,e,f,g\}$. Then for any $R\in \binom{[n]\setminus D}{k-2}$ and $S\in\binom{D}{2}$ with $R\cup S\in \hf$, we must have $S\in \hht^{(2)}(\hp)$. Recall that
 \[
 \hht^{(2)}(\hp)=\left\{\{v,d\},\{w,d\},\{w,e\},\{w,f\},\{d,f\},\{f,g\}\right\}.
 \]
 Note that $\hht^{(2)}(\hp)$ can be partitioned into three disjoint pairs: $(\{v,d\},\{w,f\})$,  $(\{w,d\},\{f,g\})$ and  $(\{w,e\},\{d,f\})$.
For any disjoint pair $\{A,A'\}\subset \hht^{(2)}(\hp)$, by \eqref{ineq-sperner} we have
\[
\alpha(A,D)+\alpha(A',D)\leq 1.
\]
Set $B=D\setminus A$ and $B'=D\setminus A'$. Using \eqref{ineq-sperner} again we infer
\begin{align*}
\alpha(A,D)+\alpha(B,D) \leq 1,\
\alpha(A',D)+\alpha(B',D) \leq 1.
\end{align*}
It follows that
\begin{align*}
&\qquad |\hf(A,D)|+|\hf(A',D)|+|\hf(B,D)|+|\hf(B',D)| \\[3pt] &=(\alpha(A,D)+\alpha(A',D))\binom{n-6}{k-2}+(\alpha(B,D)+\alpha(B',D))\binom{n-6}{k-4}\\[3pt]
&\leq \binom{n-6}{k-2}-(1-\alpha(A,D)-\alpha(A',D))\binom{n-6}{k-2}+(2-\alpha(A,D)-\alpha(A',D))\binom{n-6}{k-4}\\[3pt]
&=\binom{n-6}{k-2}+\binom{n-6}{k-4}-(1-\alpha(A,D)-\alpha(A',D))
\left(\binom{n-6}{k-2}-\binom{n-6}{k-4}\right)\\[3pt]
&\leq \binom{n-6}{k-2}+\binom{n-6}{k-4}.
\end{align*}
For any $P\in \binom{D}{2}\setminus \hht^{(2)}(\hp)$, $\hf(P,D)=\emptyset$ implies
\[
|\hf(P,D)|+|\hf(D\setminus P,D)|\leq \binom{n-6}{k-4}.
\]
For $A\in \binom{D}{3}$ just use $\alpha(A,D)+\alpha(D\setminus A,D) \leq 1$. Eventually we obtain
 \begin{align*}
 |\hf|=\sum_{Q\subset D}|\hf(Q,D)| & \leq 3\left(\binom{n-6}{k-2}+\binom{n-6}{k-4}\right)+\left(\binom{6}{4}-6\right)\binom{n-6}{k-4}+\frac{1}{2} \binom{6}{3}\binom{n-6}{k-3}\\[3pt]
 &\qquad +6\binom{n-6}{k-5}+\binom{n-6}{k-6} \\[3pt]
 &\leq 3\binom{n-6}{k-2}+10\binom{n-6}{k-3}
 +12\binom{n-6}{k-4}+6\binom{n-6}{k-5}+\binom{n-6}{k-6}.
\end{align*}
Note that
\[
\binom{n-1}{k-1} = \binom{n-6}{k-1}+5\binom{n-6}{k-2}+10\binom{n-6}{k-3}+10\binom{n-6}{k-4}+5\binom{n-6}{k-5}+\binom{n-6}{k-6}.
\]
 Then by \eqref{ineq-key14}
\[
\binom{n-1}{k-1} -|\hf| \geq \binom{n-6}{k-1}+2\binom{n-6}{k-2}-2\binom{n-6}{k-4}-\binom{n-6}{k-5}>\binom{n-k}{k-1}+\binom{n-k-1}{k-1}
\]
and thereby $|\hf|<|\hg(n,k)|$.

 {\bf\noindent Case 2.} $T=\{w,d,e\}$.

 Set $\hp=\{\{v,w,f\},\{d,e,f\},\{w,d,e\}\}\subset \hht^{(3)}(\hf)$ and $D= \{v,w,d,e,f\}$. Then for any $R\in \binom{[n]\setminus D}{k-2}$ and $S\in\binom{D}{2}$ with $R\cup S\in \hf$, we have $S\in \hht^{(2)}(\hp)$. Recall that
 \[
 \hht^{(2)}(\hp)=\left\{\{v,d\},\{v,e\},\{w,d\},\{w,e\},\{w,f\},\{d,f\},\{e,f\}\right\}.
 \]
 Note that $(\{w,d\},\{e,f\})$,
 $(\{w,e\} ,\{v,d\})$, $(\{w,f\}, \{v,e\})$ are three disjoint pairs. For a disjoint pair $\{A,A'\}$, set $B=D\setminus A$ and $B'=D\setminus A'$.
By a similar argument as in Case 1, we obtain
\begin{align*} |\hf(A,D)|+|\hf(A',D)|+|\hf(B,D)|+|\hf(B',D)|\leq \binom{n-5}{k-2}+\binom{n-5}{k-4}.
\end{align*}
For $\{d,f\}$ just use $\alpha(\{d,f\},D)+\alpha(D\setminus \{d,f\},D) \leq 1$.
This leads to
\begin{align*}
 |\hf|=\sum_{Q\subset D}|\hf(Q,D)| &\leq 4\binom{n-5}{k-2}+6\binom{n-5}{k-3}
 +5\binom{n-5}{k-4}+\binom{n-5}{k-5}.
 \end{align*}
Note that
\[
\binom{n-1}{k-1} = \binom{n-5}{k-1}+4\binom{n-5}{k-2}+6\binom{n-5}{k-3}+4\binom{n-5}{k-4}+\binom{n-5}{k-5}.
\]
Then by \eqref{ineq-key15}
\[
\binom{n-1}{k-1} -|\hf| \geq \binom{n-5}{k-1}-\binom{n-5}{k-4}>\binom{n-k}{k-1}+\binom{n-k-1}{k-1}
\]
and  $|\hf|<|\hg(n,k)|$ follows.
\end{proof}

By Propositions \ref{prop-6.5} and \ref{prop-6.9}, we conclude that Theorem \ref{thm-case1} holds.
\end{proof}

\section{Proofs of Propositions \ref{prop-key1} and \ref{cor-5.9}}

In this section, we prove Propositions \ref{prop-key1}, \ref{cor-5.9} and some inequalities that are needed in the proof of Theorem \ref{thm-main3}.

For the following lemma we need a simple inequality.
\begin{align}\label{ineq-key12}
\left(\frac{3}{2}\right)^d > 2d-3\mbox{ for } d\geq 1.
\end{align}

\begin{proof}
The statement is easily checked for $d=1,2,3$ and 4. For $d+1\geq 5$ we apply induction. As
\[
\left(\frac{3}{2}\right)^{d+1}=\frac{3}{2}\left(\frac{3}{2}\right)^d > 2d-3 +\frac{1}{2}(2d-3)>2d-3 +2=2(d+1)-3,
\]
by the induction hypothesis and $2d-3>4$ (for $d\geq 4$), the proof of \eqref{ineq-key12} is complete.
\end{proof}

Let us prove two analytic inequalities.

\begin{lem}\label{lem-5.5}
Suppose that $n,b,a$ are positive integers, $n\geq a+b$.  Then for $b\geq a+1$,
\begin{align}\label{ineq-key10}
\binom{n}{b}-2\binom{n-a}{b}+\binom{n-2a}{b}+2\geq \binom{n-1}{a-1}+\binom{n-1}{b-1}-\binom{n-a-1}{b-1}+1,
\end{align}
with equality holding iff $n=a+b$ or $a=1$.
For $b\geq a+2$,
\begin{align}\label{ineq-key11}
\binom{n}{b}-2\binom{n-a}{b}+\binom{n-2a}{b}+2\geq \binom{n}{a},
\end{align}
with equality holding iff $n=a+b$ or $a=1, b=3$.
\end{lem}

\begin{proof}
Let us prove \eqref{ineq-key10} first. For $n=a+b$, using $\binom{a+b-1}{a-1}=\binom{a+b-1}{b}$ one sees that both sides of \eqref{ineq-key10} are equal to $\binom{a+b}{b}$. In case of $a=1$, by substituting $a=1$ into \eqref{ineq-key10} and using
\[
\binom{n}{b}-2\binom{n-1}{b}+\binom{n-2}{b}=\binom{n-2}{b-2},
\]
we infer that both sides are equal to $\binom{n-2}{b-2}+2$.

From now  on we assume $n\geq a+b+1$ and $a\geq 2$. We prove \eqref{ineq-key10} with strict inequality, that is,
\begin{align}\label{ineq-new5.8}
\binom{n}{b}-2\binom{n-a}{b}+\binom{n-2a}{b}\geq \binom{n-1}{a-1}+\binom{n-1}{b-1}-\binom{n-a-1}{b-1}.
\end{align}
We distinguish two cases.

{\noindent\bf Case 1.} $b=a+1$.

Rearranging \eqref{ineq-new5.8} yields
\[
\binom{n}{a+1}-\binom{n}{a} \geq \binom{n-a}{a+1}+\binom{n-a-1}{a+1}-\binom{n-2a}{a+1}
\]
By using $\binom{n}{a+1}=\frac{n-a}{a+1}\binom{n}{a}$, $\binom{n-a}{a+1}=\frac{n-a}{a+1}\binom{n-a-1}{a}$ and $\binom{n-a-1}{a+1}=\frac{n-2a-1}{a+1}\binom{n-a-1}{a}$, it suffices to show
\begin{align}\label{ineq-new5.9}
\frac{n-2a-1}{a+1}\binom{n}{a}\geq \frac{2n-3a-1}{a+1}\binom{n-a-1}{a}-\binom{n-2a}{a+1}.
\end{align}

Plugging in $n=2a+2$ yields (note that $\binom{n-2a}{a+1}=0$)
\[
\binom{2a+2}{a}\geq (a+3)\binom{a+1}{a}=(a+1)(a+3).
\]
Equivalently, $(2a+2)!\geq  (a+1)!(a+3)!$. This is true for $a=2$. Therefore $(2a+4)(2a+3)>(a+2)(a+4)$ implies it for all $a>2$ as well.

From now on we assume that $n\geq 2a+3$ and for convenience set $t=n-2a-1$. There are two subcases according the size of $t$.

{\noindent\bf Subcase 1.1} $2\leq t\leq a$.

By \eqref{ineq-new5.9}, it suffices to show that
\begin{align}\label{ineq-new5.10}
\prod_{1\leq i\leq a}\frac{a+1+t+i}{t+i}\geq \frac{a+2t+1}{t}.
\end{align}
In the case $a=t=2$ we have $\frac{6\times 7}{3\times 4} \geq \frac{7}{2}$ and the equality holds. In the case $a=3$ we need
\[
\frac{(t+5)(t+6)(t+7)}{(t+1)(t+2)(t+3)} \geq \frac{2(t+2)}{t},
\]
which is easily verified for both $t=2$ and 3. For the case $a\geq 4$, $2\leq t\leq a$ note that the RHS of \eqref{ineq-new5.10} is maximal for $t=2$ when its value is $\frac{a+5}{2}$. On the LHS, by $t\leq a$
\[
\frac{a+1+t+i}{t+i}\geq \frac{2a+1+t}{a+t} >\frac{3}{2}.
\]
Thus \eqref{ineq-new5.10} follows from $\left(\frac{3}{2}\right)^a\geq \frac{a+5}{2}$. By \eqref{ineq-key12} and $a\geq 4$, we infer
\[
\left(\frac{3}{2}\right)^a>2a-3 > \frac{a+5}{2}
\]
and \eqref{ineq-new5.10} holds.

{\noindent\bf Subcase 1.2} $t\geq a+1$, i.e. $n\geq 3a+2$.

For $a=2$, the LHS of \eqref{ineq-new5.8} equals
\[
\binom{n}{3}-2\binom{n-2}{3}+\binom{n-4}{3}=(n-2)+2(n-3)+(n-4)=4n-12.
\]
The RHS of \eqref{ineq-new5.8} equals
\[
\binom{n-1}{1}+\binom{n-1}{2}-\binom{n-3}{2} =(n-1)+(n-2)+(n-3) =3n-6.
\]
Using $n\geq a+b+1\geq 6$, we see that \eqref{ineq-new5.8} holds. Now assume $a\geq 3$.
Recall the formula
\[
\binom{n-p}{q}-\binom{n-p-a}{q} =\sum_{1\leq j\leq a} \binom{n-p-j}{q-1}.
\]
Applying it twice we obtain
\begin{align}\label{ineq-new5.11}
\binom{n}{a+1}-2\binom{n-a}{a+1}+\binom{n-2a}{a+1}&= \sum_{1\leq i\leq a} \left(\binom{n-i}{a}-\binom{n-a-i}{a}\right)\nonumber\\[3pt]
&=\binom{n-1}{a}-\binom{n-a-1}{a}+\sum_{2\leq i\leq a} \left(\binom{n-i}{a}-\binom{n-a-i}{a}\right).
\end{align}
Using only three of the terms for $i=2$ and $i=3$ we infer that the sum in the bracket is at least
\[
\binom{n-3}{a-1}+\binom{n-4}{a-1}+\binom{n-5}{a-1}+\binom{n-4}{a-1}+\binom{n-5}{a-1}
+\binom{n-6}{a-1}>\binom{n-3}{a-1}+2\binom{n-4}{a-1}.
\]
Since $n\geq (\ell-1)a$ implies $\frac{n-a-\ell+2}{n-\ell+1}\geq \frac{n-a}{n}$, by \eqref{ineq-key1} and $n\geq 3a$ we infer
\[
\frac{\binom{n-\ell}{a-1}}{\binom{n-1}{a-1}}\geq \left(\frac{n-a-\ell+2}{n-\ell+1}\right)^{\ell-1} \geq  \left(\frac{n-a}{n}\right)^{\ell-1}\geq \left(\frac{2}{3}\right)^{\ell-1},\  \ell=3,4.
\]
Using $\left(\frac{2}{3}\right)^2+2\left(\frac{2}{3}\right)^3=\frac{28}{27}>1$, it follows that \eqref{ineq-new5.11} is more than
\[
\binom{n-1}{a}-\binom{n-a-1}{a}+\binom{n-1}{a-1},
\]
implying \eqref{ineq-new5.8}.

{\noindent\bf Case 2.} $b\geq a+2\geq 4$, $n\geq b+a+1$.

We apply induction on $n$ in which we assume that \eqref{ineq-new5.8} holds for the triples $(n-1,b,a)$ and $(n-1,b-1,a)$ to prove it for the triple $(n,b,a)$.

Let $\delta(x,y)$ be the Kronecker symbol,
\[
\delta(x,y)=\left\{
                \begin{array}{ll}
                  1, & \hbox{ if }x=y; \\[3pt]
                  0, & \hbox{ if }x\neq y.
                \end{array}
              \right.
\]
The equality we want to prove is
\[
\binom{n}{b}-2\binom{n-a}{b}+\binom{n-2a}{b}+\delta(n,a+b) \geq \binom{n-1}{a-1}+\binom{n-1}{b-1}-\binom{n-a-1}{b-1}.
\]
By the induction hypothesis and $b\geq a+2$ we may use the instances $(n-1,b,a)$ and $(n-1,b-1,a)$:
\begin{align*}
&\binom{n-1}{b}-2\binom{n-1-a}{b}+\binom{n-1-2a}{b}+\delta(n-1,a+b) \geq \binom{n-2}{a-1}+\binom{n-2}{b-1}-\binom{n-a-2}{b-1},\\[3pt]
&\binom{n-1}{b-1}-2\binom{n-1-a}{b-1}+\binom{n-1-2a}{b-1}+\delta(n-1,a+b-1) \geq \binom{n-2}{a-1}+\binom{n-2}{b-2}-\binom{n-a-2}{b-2}.
\end{align*}
Adding them we infer ($n>a+b$ implies $\delta(n-1,a+b-1)=0$)
\begin{align}\label{ineq-new5.12}
&\binom{n}{b}-2\binom{n-a}{b}+\binom{n-2a}{b}+\delta(n-1,a+b)\geq 2\binom{n-2}{a-1}+\binom{n-1}{b-1}-\binom{n-a-1}{b-1}.
\end{align}
Since $n\geq a+b>2a$, we infer
\[
2\binom{n-2}{a-1} \geq 2\frac{n-a}{n-1} \binom{n-1}{a-1}> \frac{2(n-a)}{n}\binom{n-1}{a-1}>\binom{n-1}{a-1}.
\]
Thus \eqref{ineq-key10} follows from \eqref{ineq-new5.12}.

Now we prove \eqref{ineq-key11} by a similar approach. For $n=a+b$, both sides of \eqref{ineq-key11} equal $\binom{a+b}{b}$. For $a=1$, by substituting $a=1$ into \eqref{ineq-key11} we get $\binom{n-2}{b-2}+2\geq n$. Since  $b\geq a+2\geq 3$ and $n\geq a+b$ imply $\binom{n-2}{b-2}\geq n-2$,  \eqref{ineq-key11} follows. From now  on we assume $n\geq a+b+1$, $a\geq 2$ and we prove \eqref{ineq-key11} with strict inequality. We distinguish two cases.

{\noindent\bf Case 1.} $b=a+2$.

Then \eqref{ineq-key11} is equivalent to
\begin{align*}
\binom{n}{a+2}-2\binom{n-a}{a+2}+\binom{n-2a}{a+2}+2\geq \binom{n}{a}.
\end{align*}
Note that $n\geq a+b+1\geq 2a+3$. For convenience set $t=n-2a-2$. There are two subcases according the size of $t$.

{\noindent\bf Subcase 1.1} $1\leq t\leq a-1$.

Then by $\binom{n-2a}{b} =\binom{t+2}{b}=0$, we need to show that
\begin{align}\label{ineq-new5.15}
\binom{n}{a+2}-\binom{n}{a}\geq 2\binom{n-a}{a+2}.
\end{align}
Note that
\[
\binom{n}{a}= \frac{(a+1)(a+2)}{(n-a-1)(n-a)}\binom{n}{a+2}.
\]
After rearranging \eqref{ineq-new5.15} is equivalent to
\begin{align}\label{ineq-new5.16}
\prod_{1\leq i\leq a+2}\frac{a+t+i}{t+i}\geq \frac{2(a+t+2)(a+t+1)}{(a+t+2)(a+t+1)-(a+2)(a+1)}.
\end{align}
For $a=2$ and $t=1$, we have $\frac{4\times 5\times 6\times 7}{2\times 3\times 4\times 5}=7>5=\frac{2\times 4\times 5}{4\times 5-3\times 4}$ and \eqref{ineq-new5.16} holds.
In case of $a=3$
we need
\[
\frac{(t+4)(t+5)(t+6)(t+7)(t+8)}{(t+1)(t+2)(t+3)(t+4)(t+5)} \geq \frac{2(t+5)(t+4)}{(t+5)(t+4)-20},
\]
which is easily verified for both $t=1$ and 2. Note that the RHS of \eqref{ineq-new5.16} is maximal for $t=1$ when its value is $a+3$. On the LHS, by $t\leq a-1$ we infer for $i\leq a+1$
\[
\frac{a+t+i}{t+i}\geq \frac{2a+1+t}{a+1+t} \geq \frac{3}{2}.
\]
Thus \eqref{ineq-new5.16} will follow from $\left(\frac{3}{2}\right)^{a+1}> a+3$. By \eqref{ineq-key12} and $a\geq 4$,
\[
\left(\frac{3}{2}\right)^{a+1}>2(a+1)-3 =2a-1\geq a+3.
\]

{\noindent\bf Subcase 1.2} $t\geq a$, i.e. $n\geq 3a+2$.

Note that for $a\geq 2$
\begin{align}\label{ineq-new5.17}
\binom{n}{a+2}-2\binom{n-a}{a+2}+\binom{n-2a}{a+2}
= \sum_{1\leq i\leq a} \sum_{1\leq j\leq a} \binom{n-i-j}{a}>  \binom{n-2}{a}+2\binom{n-3}{a}.
\end{align}
By \eqref{ineq-key1} and $n\geq 3a+2$, we infer for $\ell \leq 3$
\[
\frac{\binom{n-\ell}{a}}{\binom{n}{a}}\geq \left(\frac{n-a-\ell+1}{n-\ell+1}\right)^{\ell} \geq  \left(\frac{n-a-2}{n-2}\right)^{\ell}\geq \left(\frac{2}{3}\right)^{\ell},\  \ell=2,3.
\]
Using $\left(\frac{2}{3}\right)^2+2\left(\frac{2}{3}\right)^3=\frac{28}{27}>1$, we obtain that \eqref{ineq-new5.17} is greater than $\binom{n}{a}$ and \eqref{ineq-key11} holds.

{\noindent\bf Case 2.} $b\geq a+3\geq 5$, $n\geq b+a+1$.

We apply induction on $n$ in which we assume that \eqref{ineq-new5.8} holds for the triples $(n-1,b,a)$ and $(n-1,b-1,a)$ to prove it for the triple $(n,b,a)$. The equality we want to prove is
\[
\binom{n}{b}-2\binom{n-a}{b}+\binom{n-2a}{b}+2\delta(n,a+b) \geq \binom{n}{a}.
\]

By the induction hypothesis and $b\geq a+3$ we may use the instances $(n-1,b,a)$ and $(n-1,b-1,a)$:
\begin{align*}
&\binom{n-1}{b}-2\binom{n-1-a}{b}+\binom{n-1-2a}{b}+2\delta(n-1,a+b) \geq \binom{n-1}{a},\\[3pt]
&\binom{n-1}{b-1}-2\binom{n-1-a}{b-1}+\binom{n-1-2a}{b-1}+2\delta(n-1,a+b-1) \geq \binom{n-1}{a},
\end{align*}
Since $n\geq a+b+1$ implies $\delta(n-1,a+b-1)=0$,
adding them we infer
\begin{align}\label{ineq-new5.18}
&\binom{n}{b}-2\binom{n-a}{b}+\binom{n-2a}{b}+2\delta(n-1,a+b) \geq 2\binom{n-1}{a}.
\end{align}
Since $n\geq a+b> 2a+2$ and $a\geq 2$, we infer
\[
\frac{1}{a+1}\binom{n}{a} \geq \frac{1}{a+1}\binom{2a+2}{2}= 2a+1>2.
\]
It follows that
\[
2\binom{n-1}{a} \geq 2\frac{n-a}{n} \binom{n}{a}\geq  \frac{2(a+2)}{2a+2} \binom{n}{a}\geq \left(1+\frac{1}{a+1}\right)\binom{n}{a}\geq \binom{n}{a}+2.
\]
Thus \eqref{ineq-key11} follows from \eqref{ineq-new5.18} with strict inequality.
\end{proof}

Let us restate Propositions \ref{prop-key1} and \ref{cor-5.9} as  Propositions \ref{prop-key1-copy} and \ref{cor-5.9-copy}.

\begin{prop}\label{prop-key1-copy}
Let $\ha \subset\binom{[n]}{a}$ and $\hb \subset \binom{[n]}{b}$ be cross-intersecting families, $n\geq a+b$, $b>a\geq 1$. Suppose that $\ha$ is non-trivial and $\hb$ is non-empty. Then
\begin{align}\label{hahbnontrivial-copy}
  |\ha|+|\hb| \leq \binom{n}{b}-2\binom{n-a}{b}+\binom{n-2a}{b}+2.
\end{align}
Moreover, for $n>a+b$ and $a\geq 2$, $\ha_0$, $\hb_0$ are the only families
achieving equality.
\end{prop}

\begin{proof}
If $\hb$ is non-trivial then \eqref{hahbnontrivial-copy} follows from \eqref{ineq-nontrivial}. Consequently we may assume by symmetry that $1\in B$ for all $B\in \hb$. Consider the cross-intersecting families $\ha(\bar{1})$ and $\hb(1)$. They are both non-empty. Indeed, $\ha(\bar{1})\neq\emptyset$ because of $\cap \ha=\emptyset$ and $\hb(1)\neq \emptyset$ because of $|\hb(1)|=|\hb|$. Applying \eqref{ft92} and the obvious inequality $|\ha(1)|\leq \binom{n-1}{a-1}$ we infer
\[
|\ha|+|\hb| =|\ha(1)|+|\ha(\bar{1})| +|\hb(1)|\leq \binom{n-1}{a-1}+1+\binom{n-1}{b-1}-\binom{n-a-1}{b-1},
\]
By \eqref{ineq-key10} we conclude that \eqref{hahbnontrivial-copy} holds with the equality holding iff $n=a+b$ or $a=1$.
\end{proof}

\begin{prop}\label{cor-5.9-copy}
Let $\ha \subset\binom{[n]}{a}$ and $\hb \subset \binom{[n]}{b}$ be cross-intersecting families, $n\geq a+b$, $b\geq a+2\geq 3$. If $\ha$ is non-trivial, then
\begin{align}\label{hahbnontrivial2-copy}
  |\ha|+|\hb| \leq \binom{n}{b}-2\binom{n-a}{b}+\binom{n-2a}{b}+2.
\end{align}
Moreover, for $n>a+b$ and $a\geq 2$, $\ha_0$, $\hb_0$ are the only families
achieving equality.
\end{prop}

\begin{proof}
If $\hb$ is non-empty, then \eqref{hahbnontrivial2-copy} follows from \eqref{hahbnontrivial-copy}. If $\hb$ is empty, then
\[
|\ha|+|\hb|=|\ha| \leq \binom{n}{a}.
\]
By  \eqref{ineq-key11} we obtain \eqref{hahbnontrivial2-copy} with the equality holding iff $n=a+b$ or $a=1, b=3$.
\end{proof}

\begin{cor}
For $k\geq 5$ and $n\geq 2k+1$,
\begin{align}\label{ineq-key8}
\binom{n-4}{k-1}-2\binom{n-k-2}{k-1}+\binom{n-2k}{k-1}>\binom{n-5}{k-2}-\binom{n-k-3}{k-2}+\binom{n-5}{k-4}
\end{align}
\end{cor}

\begin{proof}
Applying \eqref{ineq-key10} for  $(n-4,k-1,k-2)$, we obtain that
\[
\binom{n-4}{k-1}-2\binom{n-k-2}{k-1}+\binom{n-2k}{k-1}\geq \binom{n-5}{k-3}+\binom{n-5}{k-2}-\binom{n-k-3}{k-2}-1.
\]
Since $k-4<k-3< n-5$ and $(n-5)> 2(k-3)$, we infer
\[
\binom{n-5}{k-3}-\binom{n-5}{k-4} = \frac{n-2k+2}{k-3}\binom{n-5}{k-4}\geq \frac{3(n-5)}{k-3}>1.
\]
That is, $\binom{n-5}{k-3}> \binom{n-5}{k-4}+1$ and \eqref{ineq-key8} follows.
\end{proof}



\begin{lem}
For $n> 2k$, $k\geq 5$ and $2\leq q\leq k-1$,
\begin{align}\label{ineq-key9}
\binom{n-q-1}{k-1} - 2\binom{n-k-1}{k-1}+\binom{n-2k+q-1}{k-1}+2> \binom{n-q-1}{k-q}-\binom{n-k-q-1}{k-q}.
\end{align}
\end{lem}

\begin{proof}
For $3\leq q\leq k-1$, by applying \eqref{ineq-key11} to the triple $(n-q-1,k-1,k-q)$ and noting $n-q-1> (k-1)+(k-q)$, we obtain
\[
\binom{n-q-1}{k-1} - 2\binom{n-k-1}{k-1}+\binom{n-2k+q-1}{k-1}+2> \binom{n-q-1}{k-q}
\]
and \eqref{ineq-key9} follows.

Now we assume that $q=2$ and \eqref{ineq-key9} is equivalent to
\begin{align}\label{ineq-neq5.22}
\binom{n-3}{k-1} - 2\binom{n-k-1}{k-1}+\binom{n-2k+1}{k-1}+2\geq \binom{n-3}{k-2}-\binom{n-k-3}{k-2}.
\end{align}
 Let
\[
h(n,k) = \binom{n-3}{k-1} - 2\binom{n-k-1}{k-1}+\binom{n-2k+1}{k-1}.
\]
We distinguish two cases.

{\noindent\bf Case 1. } $n\geq 3k$.

Note that for $k\geq 5$
\begin{align*}
h(n,k) &= \sum_{1\leq i\leq k-2} \left(\binom{n-3-i}{k-2} -\binom{n-k-i-1}{k-2} \right)\\[3pt]
&\geq  \binom{n-4}{k-2} -\binom{n-k-2}{k-2} +\binom{n-5}{k-2} -\binom{n-k-3}{k-2}+\binom{n-6}{k-2} -\binom{n-k-4}{k-2}.
\end{align*}
Since $n\geq (\ell+2)k/2$ implies $\frac{n-k-\ell}{n-2-\ell} \geq \frac{n-k}{n}$,  by \eqref{ineq-key2} we infer
\[
\frac{\binom{n-3-\ell}{k-2} -\binom{n-k-1-\ell}{k-2}}{\binom{n-3}{k-2}-\binom{n-k-1}{k-2}}\geq
\left(\frac{n-k-\ell}{n-2-\ell}\right)^\ell\geq \left(\frac{n-k}{n}\right)^\ell, \ \ell=1,2.
\]
Thus, for $n\geq 3k$ and $k\geq 5$ we have
\begin{align*}
h(n,k,q)&\geq \left(\frac{2}{3}+\left(\frac{2}{3}\right)^2\right)\left(\binom{n-3}{k-2}-\binom{n-k-1}{k-2}\right)+\binom{n-6}{k-2} -\binom{n-k-4}{k-2}\\[3pt]
&> \binom{n-3}{k-2}-\binom{n-k-1}{k-2}+\binom{n-6}{k-2} -\binom{n-k-4}{k-2}.
\end{align*}
Using $\binom{n-6}{k-2}\geq \binom{n-k-1}{k-2}$ and $\binom{n-k-4}{k-2}\leq \binom{n-k-3}{k-2}$, we conclude that
\begin{align*}
h(n,k)&> \binom{n-3}{k-2}-\binom{n-k-1}{k-2}+\binom{n-k-1}{k-2} - \binom{n-k-3}{k-2}=\binom{n-3}{k-2} - \binom{n-k-3}{k-2}
\end{align*}
and \eqref{ineq-neq5.22} holds.

{\noindent\bf Case 2. } $2k< n\leq 3k-1$.

To prove \eqref{ineq-neq5.22}, it suffices to show that
\[
\binom{n-3}{k-1} - 2\binom{n-k-1}{k-1}\geq \binom{n-3}{k-2}-\binom{n-k-3}{k-2}.
\]
Since $\binom{n-3}{k-2}=\frac{k-1}{n-k-1}\binom{n-3}{k-1}$ and $\binom{n-k-3}{k-2}=\frac{(k-1)(n-2k)}{(n-k-1)(n-k-2)}\binom{n-k-1}{k-1}$, the equality is equivalent to \[
\frac{n-2k}{n-k-1}\binom{n-3}{k-1} \geq \left(2-\frac{(k-1)(n-2k)}{(n-k-1)(n-k-2)}\right)\binom{n-k-1}{k-1}.
\]
Let $t=n-2k$. Then $1\leq t\leq k-1$. Substituting $t=n-2k$ and expanding the binomial coefficients, we get the equivalent version
\[
\frac{t}{t+k-1}\prod_{1\leq i\leq k-1}\frac{t+i+k-2}{t+i} \geq 2-\frac{(k-1)t}{(t+k-1)(t+k-2)}.
\]
By moving the first term $\frac{t+k-1}{t+1}$ in front of the $\prod$,  it changes to:
\begin{align}\label{ineq-new5.23}
\frac{t}{t+1}\prod_{2\leq i\leq k-1}\frac{t+i+k-2}{t+i}\geq 2-\frac{(k-1)t}{(t+k-1)(t+k-2)}.
\end{align}
Note that  $2\leq i\leq k-3$ and $t\leq k-1$ imply $\frac{t+i+k-2}{t+i}\geq \frac{3k-6}{2k-4}=\frac{3}{2}$. Moreover, $\frac{t}{t+1}\geq \frac{1}{2}$ and $\frac{t+2k-4}{t+k-2}\frac{t+2k-3}{t+k-1} \geq \frac{3k-5}{2k-3}\frac{3k-4}{2k-2}$. Consequently for $k\geq 5$
\[
\frac{t}{t+1}\frac{t+2k-4}{t+k-2}\frac{t+2k-3}{t+k-1} \geq \frac{1}{2}\cdot \frac{3k-5}{2k-3}\cdot\frac{3k-4}{2k-2} \geq \frac{ 10\times 11}{2\times7\times 8}=\frac{55}{56}.
\]
It follows that the LHS of \eqref{ineq-new5.23} is greater than $\frac{55}{56} \left(\frac{3}{2}\right)^{k-4}$. On the other hand, the RHS of \eqref{ineq-new5.23} is less than $2$. For $k\geq 6$ we have
\[
\frac{55}{56} \left(\frac{3}{2}\right)^{k-4} \geq \frac{55}{56} \left(\frac{3}{2}\right)^2 = \frac{495}{224}>2.
\]
Thus  \eqref{ineq-new5.23} holds for $k\geq 6$.
For $k=5$ and  $1\leq t\leq 4$, one can check directly that \eqref{ineq-new5.23} holds. Thus the lemma is proven.
\end{proof}

\section{The proof of Theorem \ref{thm-main3}}

In this section we use the method of shifting ad extremis, introduced in Section 2 to obtain the exact value of $f(n,k,3)$ for $n> 2k$ and $k\geq 7$. Before that let us prove the following two inequalities.

\begin{lem}
For $n> 2k\geq 6$,
\begin{align}\label{ineq-5.4}
g(n,k,3)< |\hg(n,k)|.
\end{align}
\end{lem}

\begin{proof}
Note that
\[
g(n,k,3) = \binom{n-1}{k-1}-\binom{n-k-2}{k-1}-(k+1)\binom{n-k-2}{k-2}+k+1.
\]
By \eqref{ineq-5.2} we have
\begin{align*}
 |\hg(n,k)|-g(n,k,3)
&= k\binom{n-k-2}{k-2}-\left(\binom{n-k}{k-1}-\binom{n-2k}{k-1}\right)+\binom{n-k-2}{k-3}-k+2\\[3pt]
&=(k-2)\left(\binom{n-k-2}{k-2}-1\right)-\binom{n-k-3}{k-2}-\cdots-\binom{n-2k}{k-2}\\[3pt]
&=\sum_{1\leq i\leq k-2}\left(\binom{n-k-2}{k-2}-1-\binom{n-k-2-i}{k-2}\right)> 0.
\end{align*}
\end{proof}

\begin{lem}
For $n> 2k\geq 6$,
\begin{align}\label{ineq-key6}
|\hg(n,k)|&> 3\left(\binom{n-4}{k-2}-\binom{n-k-2}{k-2}+1\right)+4\binom{n-4}{k-3}+\binom{n-4}{k-4}.
\end{align}
\end{lem}

\begin{proof}
Let $R=(5,6,\ldots,k+2)$. Consider the following construction.
\begin{align*}
\hf_R=&\left\{P\cup R\colon P\in \binom{\{2,3,4\}}{2}\right\}\bigcup \left\{F\in \binom{[n]}{k}\colon |F\cap [4]|\geq 3\right\}\\[3pt]
& \quad \bigcup \left\{F\in \binom{[2,n]}{k}\colon F\cap [4]=(1,2) \mbox{ or }(1,3) \mbox{ or }(1,4), \ F\cap R\neq \emptyset\right\}.
\end{align*}
It is easy to check that $\hf_R$ is intersecting, initial and $\tau(\hf)=3$. Moreover,
\[
|\hf_R|=3\left(\binom{n-4}{k-2}-\binom{n-k-2}{k-2}+1\right)+4\binom{n-4}{k-3}+\binom{n-4}{k-4}.
\]
Then by \eqref{ineq-5.4} we conclude that
\[
|\hf_R|\leq g(n,k,3)< |\hg(n,k)|.
\]
\end{proof}

Now we are in a position to prove our main theorem.

\begin{proof}[Proof of Theorem \ref{thm-main3}]
Assume that $\hf\subset \binom{[n]}{k}$ is an intersecting family, $n> 2k$, $\tau(\hf)\geq 3$ and $|\hf|$ is maximal. Without loss of generality, assume further that $\hf$ is shifted ad extremis with respect to $\tau(\hf)\geq 3$ and let $\mathds{H}$ be the corresponding shift-resistant graph. If $\hf$ is initial, then by \eqref{ineq-5.4} $|\hf|\leq  g(n,k,3)< |\hg(n,k)|$ and we are done. Thus we may assume $\mathds{H}\neq \emptyset$.

By Theorem \ref{thm-case1}, we assume that $\hht^{(3)}(\hf)$ is a star. Without loss of generality, suppose that all  $T\in \hht^{(3)}(\hf)$  contain $a$. Let $\ha=\hf(a)$, $\hb=\hf(\bar{a})$ and set $V=[n]\setminus \{a\}$. By Lemma \ref{lem2.1}, $\ha,\hb$ are shifted ad extremis with respect to $\tau(\hb)\geq 2$ and $\mathds{H}\cap \binom{V}{2}$ is the corresponding shift-resistant graph.  Define the 2-cover graph $\hhat$ as
\[
\hhat=\left\{(i,j)\in \binom{V}{2}\colon \hb(\bar{i},\bar{j}) =\emptyset\right\}.
\]
By Claim \ref{claim-6.4}, $\mathds{H}\cap \binom{V}{2}$ is non-empty. Since $\mathds{H}\cap \binom{V}{2}$ is a subgraph of $\hhat$, $\hhat$ is non-empty as well.

Now by Proposition \ref{prop-5.4}, either $\hhat$ contains a triangle or $\hhat$ is a complete bipartite graph on partite sets $X$ and $Y$ where $X\cup Y$ is the set of the first $|X|+|Y|$ elements in $V$, $2\leq |X|\leq k$, $2\leq |Y|\leq k$.

We deal with the two cases separately.

\begin{prop}\label{lem7.5}
 If  $\hhat$ contains a triangle, $n\geq 2k+1$ and $k\geq 5$, then
\[
|\hf|=|\ha|+|\hb| \leq |\hg(n,k)|.
\]
Moreover, the equality holds iff $\hf=\hg(n,k)$ up to isomorphism.
\end{prop}
\begin{proof}
Let $(u,v,w)\subset V$ be a triangle in $\hhat$ with $u+v+w$ minimal.
By the definition of $\hhat$, we know that
\[
\hb(\bar{u},\bar{v}) = \hb(\bar{u},\bar{w}) =\hb(\bar{v},\bar{w}) =\emptyset.
\]
Since $\ha,\hb$ are saturated cross-intersecting with respect to $\ha$,  $(u,v)$, $(v,w)$, $(u,w)$ are full in $\ha$.

\begin{claim}\label{lem-6.13}
For every $x\in \{u,v,w\}$ there is $A_x\in \ha$ satisfying
\begin{align}\label{claim3}
A_x \cap\{u,v,w\}=\{x\}.
\end{align}
\end{claim}
\begin{proof}
Suppose that \eqref{claim3} fails for $x$.  Since  $\{u,v,w\}\setminus \{x\}$ is a cover of $\hb$, $\tau(\hf)>2$  implies that $A\cap \{u,v,w\}=\emptyset$ for some $A\in \ha$.

We claim that $\{z,x\}$ is shiftable for every $z\in A$. If $\{z,x\}$ is a shift-resistant pair, then $\hb(\bar{z},\bar{x})=\emptyset$.  Since $(\ha,\hb)$ forms a saturated pair, it follows that $\ha(z,x)$ is full. Then there are many $A_x\in \ha$ satisfying \eqref{claim3}.  Thus $\{z,x\}$ must be shiftable.

Should $z<x$ hold, $S_{zx}(\ha)=\ha$ and the fullness of each edge of $\{u,v,w\}$ imply that $(\{u,v,w\}\setminus\{x\})\cup \{z\}$ is also a triangle. This contradicts the minimality of $u+v+w$. Thus $z>x$ for every $z\in A$.

Now fix $z\in A$. As $\{z,x\}$ is shiftable,  $S_{xz}(A)=(A\setminus\{z\})\cup \{x\}$ is in $\ha$ and satisfies \eqref{claim3}, concluding the proof.
\end{proof}

Recall that $(u,v)$, $(v,w)$, $(u,w)$ are full in $\ha$. By the cross-intersecting property $|B\cap \{u,v,w\}|\geq 2$ for all  $B\in \hb$. For
notational convenience, set
\[
\hb_{\bar{x}}=\hb(\{u,v,w\}\setminus \{x\},\{u,v,w\})\subset \binom{V\setminus \{u,v,w\}}{k-2}
\]
and
\[
\ha_x = \ha(\{x\},\{u,v,w\}) \subset \binom{V\setminus \{u,v,w\}}{k-2},\ x\in \{u,v,w\}.
\]
Note that $\ha_x$, $\hb_{\bar{x}}$ are cross-intersecting. Since $\hb$ is non-trivial, $\hb_{\bar{x}}$ is non-empty. By \eqref{claim3} $\ha_x$ is also non-empty. Thus by \eqref{ft92} we have
\begin{align}\label{ineq-hauhbubar}
|\ha_x|+|\hb_{\bar{x}}| \leq \binom{n-4}{k-2}-\binom{n-k-2}{k-2}+1.
\end{align}

Set
\[
\ha_0 = \{A\in \ha\colon |A\cap \{u,v,w\}|\geq 2\}.
\]
By fullness, we have
\begin{align}\label{ineq-ha0tcase}
|\ha_0| = 3\binom{n-4}{k-3}+\binom{n-4}{k-4}.
\end{align}

If $\ha(\overline{\{u,v,w\}})=\emptyset$,  then
\begin{align}\label{ineq-6.2}
|\ha(\overline{\{u,v,w\}})|+|\hb(\{u,v,w\})| =|\hb(\{u,v,w\})|\leq \binom{n-4}{k-3}.
\end{align}
Adding \eqref{ineq-ha0tcase}, \eqref{ineq-6.2} and \eqref{ineq-hauhbubar} for $x=u,v,w$, we obtain that
\begin{align*}
|\ha|+|\hb|&=|\ha_0|+ |\ha(\overline{\{u,v,w\}})|+|\hb(\{u,v,w\})| +\sum_{x\in\{u,v,w\}}\left(|\ha_x|+|\hb_{\bar{x}}|\right) \\[3pt]
&\leq 4\binom{n-4}{k-3}+\binom{n-4}{k-4}+ 3\left(\binom{n-4}{k-2}-\binom{n-k-2}{k-2}+1\right)\\[5pt]
&\overset{\eqref{ineq-key6}}{<}|\hg(n,k)|
\end{align*}
and we are done. Thus in the rest of the proof we may assume that $\ha(\overline{\{u,v,w\}})\neq\emptyset$.

\begin{claim}\label{claim5-5}
If $\ha(\overline{\{u,v,w\}})\neq\emptyset$, then
\begin{align}\label{ineq-ha123bar}
|\ha(\overline{\{u,v,w\}})|+|\hb(\{u,v,w\})|\leq \binom{n-4}{k-1}-2\binom{n-k-2}{k-1}+\binom{n-2k}{k-1}.
\end{align}
Moreover, the equality holds iff $\hb(\{u,v,w\})=\emptyset$ and there exist disjoint sets $S,T\in \binom{V\setminus \{u,v,w\}}{k-2}$ such that $\hb_{\bar{u}}\cup \hb_{\bar{v}}\cup \hb_{\bar{w}}=\{S,T\}$ and
$\ha(\overline{\{u,v,w\}})=\{E\in \binom{V\setminus \{u,v,w\}}{k-1}\colon E\cap S\neq \emptyset,  E\cap T\neq \emptyset\}$.
\end{claim}

\begin{proof}
We distinguish two cases.

\vspace{3pt}
{\noindent \bf Case 1.} $\ha(\overline{\{u,v,w\}})$ is non-trivial.
\vspace{3pt}

If $\hb(\{u,v,w\})=\emptyset$ then let $\hb^* =\hb_{\bar{u}}\cup \hb_{\bar{v}}\cup \hb_{\bar{w}}$. The non-triviality of $\hb$ implies the non-triviality of $\hb^*$ and thereby $|\hb^*|\geq 2$. By \eqref{ineq-nontrivial}, we obtain that
\[
|\ha(\overline{\{u,v,w\}})| \leq |\ha(\overline{\{u,v,w\}})|+|\hb^*|-2 \leq \binom{n-4}{k-1}  -2\binom{n-k-2}{k-1}+\binom{n-2k}{k-1}.
\]
Moreover, the equality holds iff there exist disjoint sets $S,T\in \binom{V\setminus \{u,v,w\}}{k-2}$ such that $\hb^*=\{S,T\}$ and
$\ha(\overline{\{u,v,w\}})=\{E\in \binom{V\setminus \{u,v,w\}}{k-1}\colon E\cap S\neq \emptyset,  E\cap T\neq \emptyset\}$.

If $\hb(\{u,v,w\})\neq \emptyset$ then let
\[
\hb^+ =\left\{D\in \binom{V\setminus \{u,v,w\}}{k-2}\colon \exists B\in \hb(\{u,v,w\}),\ B\subset D\right\} \mbox{ and }\hat{\hb}=\hb^*\cup \hb^+.
\]
Clearly $|\hb^+|\geq n-4-(k-3)$. We claim that $|\hat{\hb}|\geq  |\hb(\{u,v,w\})|+2$. For $n\geq 2k+1$, we have $n-4-(k-3)\geq k$. Consequently for $|\hb(\{u,v,w\})|\leq k-2$, we have
\[
|\hat{\hb}|\geq |\hb^+|\geq k\geq  |\hb(\{u,v,w\})|+2.
\]
From now on assume $|\hb(\{u,v,w\})|\geq k-1$. By Sperner's argument \cite{Sperner},
\[
 |\hb^+|\geq |\hb(\{u,v,w\})|\frac{n-4-(k-3)}{k-2}.
\]
Set $t=|\hb(\{u,v,w\})|$. Since $n-4-(k-3)\geq k$, it is sufficient to show
\[
\frac{tk}{k-2}\geq t+2 \mbox{ or equivalently } 2t\geq 2(k-2),
\]
which follows from $t\geq k-1>k-2$. Thus $|\hat{\hb}|\geq  |\hb(\{u,v,w\})|+2\geq 3$.

Note that $\ha(\overline{\{u,v,w\}})$ and $\hat{\hb}$ are cross-intersecting and both non-trivial. By \eqref{ineq-nontrivial} and $|\hat{\hb}|\geq 3$, we obtain
\[
|\ha(\overline{\{u,v,w\}})|+|\hb(\{u,v,w\})|\leq |\ha(\overline{\{u,v,w\}})|+|\hat{\hb}|-2 < \binom{n-4}{k-1}  -2\binom{n-k-2}{k-1}+\binom{n-2k}{k-1}.
\]

\vspace{3pt}
{\noindent \bf Case 2.} $\ha(\overline{\{u,v,w\}})$  is a star.
\vspace{3pt}

Let $z\in \cap \ha(\overline{\{u,v,w\}})$. Clearly, $|\hb(\{u,v,w,z\})|\leq \binom{n-5}{k-4}$. If $\hb(\{u,v,w\},\bar{z})\neq \emptyset$, then by \eqref{ft92} we have
\[
|\ha(\overline{\{u,v,w\}},z)|+|\hb(\{u,v,w\},\bar{z})|\leq \binom{n-5}{k-2}-\binom{n-5-(k-3)}{k-2}+1.
\]
It follows that
\begin{align*}
|\ha(\overline{\{u,v,w\}})|+|\hb(\{u,v,w\})|&=|\ha(\overline{\{u,v,w\}},z)|+
|\hb(\{u,v,w\},\bar{z})|+|\hb(\{u,v,w,z\})|\\[3pt]
&\leq \binom{n-5}{k-2}-\binom{n-5-(k-3)}{k-2}+1+\binom{n-5}{k-4}\\[3pt]
&\leq \binom{n-5}{k-2}-\binom{n-k-3}{k-2}+\binom{n-5}{k-4}\\[3pt]
&\overset{\eqref{ineq-key8}}< \binom{n-4}{k-1}-2\binom{n-k-2}{k-1}+\binom{n-2k}{k-1}.
\end{align*}

If $\hb(\{u,v,w\},\bar{z})= \emptyset$, then $|\hb(\{u,v,w\})|=|\hb(\{u,v,w,z\})|\leq \binom{n-5}{k-4}$. Since $\hb$ is non-trivial, we may choose $B_z\in \hb$ such that $z\notin B_z$. By the cross-intersecting property, $A\cap B_z\neq \emptyset$ for all $A\in \ha(\overline{\{u,v,w\}})$. Note that $|B_z\cap \{u,v,w\}|\geq 2$. Thus,
\begin{align*}
|\ha(\overline{\{u,v,w\}})|+|\hb(\{u,v,w\})|&=|\ha(\overline{\{u,v,w\}},z)|+|\hb(\{u,v,w,z\})|\\[3pt]
&\leq \binom{n-5}{k-2}-\binom{n-5-|B_z\setminus \{u,v,w\}|}{k-2}+\binom{n-5}{k-4}\\[3pt]
&\leq \binom{n-5}{k-2}-\binom{n-k-3}{k-2}+\binom{n-5}{k-4}\\[3pt]
&\overset{\eqref{ineq-key8}}<\binom{n-4}{k-1}-2\binom{n-k-2}{k-1}+\binom{n-2k}{k-1}.
\end{align*}
\end{proof}

Adding \eqref{ineq-ha0tcase}, \eqref{ineq-hauhbubar} for $x=u,v,w$ and \eqref{ineq-ha123bar}, we arrive at
\begin{align*}
|\ha|+|\hb|&=|\ha_0|+\sum_{x\in\{u,v,w\}}\left(|\ha_x|+|\hb_{\bar{x}}|\right) + |\ha(\overline{\{u,v,w\}})|+|\hb(\{u,v,w\})| \\[3pt]
&\leq 3\binom{n-4}{k-3}+\binom{n-4}{k-4}+ 3\binom{n-4}{k-2}-3\binom{n-k-2}{k-2}+3\\[3pt]
&\qquad +\binom{n-4}{k-1}-2\binom{n-k-2}{k-1}+\binom{n-2k}{k-1}\\[3pt]
&=\binom{n-1}{k-1}-2\binom{n-k-1}{k-1}-\binom{n-k-2}{k-2}+\binom{n-2k}{k-1}+3\\[3pt]
&=\binom{n-1}{k-1}-\binom{n-k-1}{k-1}-\binom{n-k}{k-1}+\binom{n-2k}{k-1}+\binom{n-k-2}{k-3}+3\\[3pt]
&=|\hg(n,k)|.
\end{align*}
The equality holds iff equalities hold in \eqref{ineq-ha123bar} and \eqref{ineq-hauhbubar} for $x=u,v,w$. By Claim \ref{claim5-5}, $\hb(\{u,v,w\})=\emptyset$ and there exist disjoint sets $S,T\in \binom{V\setminus \{u,v,w\}}{k-2}$ such that $\hb_{\bar{u}}\cup \hb_{\bar{v}}\cup \hb_{\bar{w}}=\{S,T\}$ and
$\ha(\overline{\{u,v,w\}})=\{E\in \binom{V\setminus \{u,v,w\}}{k-1}\colon E\cap S\neq \emptyset,  E\cap T\neq \emptyset\}$. By Theorem \ref{thm-ft92} and $|\hb_{\bar{u}}\cup \hb_{\bar{v}}\cup \hb_{\bar{w}}|=2$, we infer that $|\hb_{\bar{u}}|=|\hb_{\bar{v}}|=|\hb_{\bar{w}}|=1$. Without loss of generality, assume that
\[
\hb=\left\{S\cup \{u,v\}, S\cup\{u,w\}, T\cup \{v,w\}\right\}.
\]
Then
\[
\ha=\left\{A\in \binom{V}{k-1}\colon A\cap B\neq \emptyset \mbox{ for all }B\in \hb\right\}.
\]
Thus the proposition is proven.
\end{proof}

\begin{prop}\label{lem7.6}
Suppose that $n> 2k$ and $k\geq 5$. If $\hhat$ is a complete bipartite graph on partite sets $X$ and $Y$ where $X\cup Y$ consists of the  first $|X|+|Y|$ elements in $V$, $2\leq |X|\leq k$, $2\leq |Y|\leq k$, then
\[
|\hf|=|\ha|+|\hb| < |\hg(n,k)|.
\]
\end{prop}
\begin{proof}
Assume that  $X=\{x_1,\ldots,x_p\}$, $Y=\{y_1,\ldots,y_q\}$ and $\{x_1,\ldots,x_p,y_1,\ldots,y_q\}$ is the set of first $p+q$ elements in $V$.
Define
\begin{align*}
&\ha_0=\left\{A\in \binom{V}{k-1}\colon A\cap X\neq \emptyset,A\cap Y\neq \emptyset\right\}.
\end{align*}
By Claim \ref{claim-6.7}, we know that for any $B\in \hb$ either $X\subset B$ or $Y\subset B$.
 By saturatedness, we have
\begin{align}\label{ineq-ha0}
|\ha_0| = \binom{n-1}{k-1}-\binom{n-p-1}{k-1}-\binom{n-q-1}{k-1}+\binom{n-p-q-1}{k-1}.
\end{align}

\begin{claim}\label{claim1}
$p+q\geq k+2$.
\end{claim}
\begin{proof}
Suppose for contradiction  that $p+q\leq k+1$. Let $K$ be the set of first $k+1$ elements in $V$. We first show that $\binom{K}{k}\subset \hb$. Since $\hb$ is shifted on $V\setminus X$ and $V\setminus Y$, it is sufficient to show that $K\setminus \{x_1\}$ and $K\setminus \{y_1\}$ are in $\hb$, where $x_1$ ($y_1$) is the smallest element of $X$ ($Y$), respectively. By symmetry consider $x_1$. By non-triviality of $\hb$, there exists $B\in \hb$ with $x_1\notin B$. This implies $B\cap Y=Y$. Now $B\setminus Y\in \binom{V\setminus Y}{k-q}$. Set $B_1=K\setminus \{x_1\}$. By $Y\subset X\cup Y\subset K$, $Y\subset B_1$, we infer that $B_1\setminus Y\in \binom{V\setminus Y}{k-q}$ and $B_1\setminus Y\prec B\setminus Y$. By shiftedness on $V\setminus Y$, $B_1\in \hb$. Combining with $K\setminus \{y_1\}\in \hb$, $\binom{K}{k}\subset \hb$ follows.

 Recall that $\mathds{H}\cap\binom{V}{2}\neq \emptyset$. Let $(i,j)\in \mathds{H}\cap\binom{V}{2}$. By  Claim \ref{claim-6.4}, we know that $S_{ij}(\hb)$ is a star. Since $\mathds{H}\cap \binom{V}{2}$ is a subgraph of $\hhat$, $(i,j)\in \hhat$. It follows  that $(i,j)$ is in $X\times Y$, i.e., either $i\in X$, $j\in Y$ or $i\in Y$, $j\in X$. In any case, $i<j$ and $i,j\in X\cup Y\subset K$. Consequently $K\setminus \{j\}$ and $K\setminus \{i\}$ are in $\hb$. Thus $K\setminus\{i,j\}\in \hb(i)\cap \hb(j)$, contradicting the fact that $S_{ij}(\hb)$ is a star.
\end{proof}

\begin{claim}\label{claim2}
If $p\leq k-1$ then $\hb(X)$ is non-trivial. If $q\leq k-1$ then $\hb(Y)$ is non-trivial.
\end{claim}

\begin{proof}
Suppose the contrary and let $z\in \tilde{B}$ for all $\tilde{B}\in \hb(X)$. Note that $z\notin Y$. Indeed the opposite would mean $z\in \hb$ for all $B\in \hb$, i.e., $\cap \hb\neq \emptyset$, contradiction.

Thus $z\notin Y$. Consequently $\{z,y\}$ is a transversal of $\hb$ for all $y\in Y$.  That is, $(X\cup \{z\})\times Y$ should be the 2-cover graph $\hhat$, contradiction.
\end{proof}

Now we distinguish two cases.

\vspace{3pt}
{\noindent\bf Case 1.} $p,q\leq k-1$.
\vspace{3pt}

Consider $\ha(\overline{X}) \subset \binom{V\setminus X}{k-1}$ and $\hb(X) \subset \binom{V\setminus X}{k-p}$.
Using $p,q\leq k-1$,  Claim \ref{claim1} implies $p,q\geq 3$. By Claim \ref{claim2}, both $\hb(X)$ and $\hb(Y)$ are non-trivial. Note that $\hb(X)$ is $(k-p)$-uniform,  $\ha(\overline{X})$ is $(k-1)$-uniform. Since $(k-1)\geq (k-p)+2$ and $n-p-1\geq (k-1)+(k-p)$, by Proposition \ref{cor-5.9} we infer
\begin{align}
|\ha(\overline{X})|+|\hb(X)| &\leq \binom{n-p-1}{k-1} - 2\binom{n-k-1}{k-1}+\binom{n-2k+p-1}{k-1}+2,\label{ineq-hahbx}
\end{align}
Similarly, we have
\begin{align}
|\ha(\overline{Y})|+|\hb(Y)| &\leq  \binom{n-q-1}{k-1} - 2\binom{n-k-1}{k-1}+\binom{n-2k+q-1}{k-1}+2.\label{ineq-hahby}
\end{align}

By adding \eqref{ineq-ha0}, \eqref{ineq-hahbx} and \eqref{ineq-hahby}, we arrive at
\begin{align*}
|\ha|+|\hb|&\leq |\ha_0|+|\ha(\overline{X})|+|\hb(X)|+|\ha(\overline{Y})|+|\hb(Y)|\\[3pt]
&\leq \binom{n-1}{k-1}+\binom{n-p-q-1}{k-1}-4\binom{n-k-1}{k-1}+\binom{n-2k+p-1}{k-1}+\binom{n-2k+q-1}{k-1}+4.
\end{align*}
Let $p+q=x$ and define
\begin{align*}
f(x,p) &=\binom{n-1}{k-1}+\binom{n-x-1}{k-1}-4\binom{n-k-1}{k-1}+\binom{n-2k+p-1}{k-1}+\binom{n-2k+x-p-1}{k-1}+4.
\end{align*}
Since
\[
\binom{x-1}{\ell-1} = \frac{(x-1)(x-2)\cdots(x-\ell+1)}{(\ell-1)!} < \frac{x(x-1)\cdots(x-\ell+2)}{(\ell-1)!}=\binom{x}{\ell-1},
\]
we see  $\binom{x-1}{\ell-1}< \binom{x}{\ell-1}$ for $x>\ell-2$.
It follows that $2\binom{x}{\ell} <  \binom{x-1}{\ell}+\binom{x+1}{\ell}$.
Hence $2f(x,p)< f(x,p+1)+f(x,p-1)$ for fixed $x$ and $x-k+1< p< k-1$.
 By symmetry we have $f(x,k-1)=f(x,x-k+1)$. Therefore,
\[
f(x,p)\leq \max\{f(x,k-1),f(x,x-k+1)\} =f(x,k-1).
\]
Note also that $2f(x,p)< f(x-1,p)+f(x+1,p)$ for fixed $p$ and $k+1<k+2\leq x=p+q\leq 2k-2$, we deduce that
\[
f(x,p) \leq \max\left\{f(k+1,k-1), f(2k-2,k-1)\right\}.
\]
Note that
\begin{align*}
f(k+1,k-1) &= \binom{n-1}{k-1} +\binom{n-k-2}{k-1}-4\binom{n-k-1}{k-1}+\binom{n-k-2}{k-1} +\binom{n-2k+1}{k-1}+4 \\[3pt]
& =\binom{n-1}{k-1} +\binom{n-k-2}{k-1}-\binom{n-k-1}{k-1}-\binom{n-k}{k-1}+\binom{n-k-1}{k-2}\\[3pt]
&\qquad-2\binom{n-k-1}{k-1}+\binom{n-2k}{k-1}+\binom{n-2k}{k-2}+\binom{n-k-2}{k-1}+4 \\[3pt]
&=|\hg(n,k)|+\binom{n-2k}{k-2}+1-\binom{n-k-2}{k-2}<|\hg(n,k)|
\end{align*}
and
\begin{align*}
f(2k-2,k-1) &= \binom{n-1}{k-1} +\binom{n-2k+1}{k-1}-4\binom{n-k-1}{k-1}+2\binom{n-k-2}{k-1}+4 \\[3pt]
& =\binom{n-1}{k-1} +\binom{n-2k}{k-1}+\binom{n-2k}{k-2}-\binom{n-k-1}{k-1}-\binom{n-k}{k-1}\\[3pt]
&\qquad+\binom{n-k-1}{k-2}-2\binom{n-k-1}{k-1}
+2\binom{n-k-2}{k-1}+4 \\[3pt]
&=|\hg(n,k)|+\binom{n-2k}{k-2}+1-\binom{n-k-2}{k-2}<|\hg(n,k)|.
\end{align*}
Thus we conclude that $|\ha|+|\hb|< |\hg(n,k)|$.

\vspace{3pt}
{\noindent\bf Case 2.} $p=k$.
\vspace{3pt}

Then $X\in \hb$ and the intersection property of $\hb$ implies that $q\leq k-1$. Since $\ha,\hb$ are cross-intersecting,  we infer $\ha(\overline{X})=\emptyset$. It follows that
\begin{align}\label{ineq-hahbx2}
|\ha(\overline{X})|+|\hb(X)| =1.
\end{align}

\begin{claim}
\begin{align}
|\ha(\overline{Y})|+|\hb(Y)| &<  \binom{n-q-1}{k-1} - 2\binom{n-k-1}{k-1}+\binom{n-2k+q-1}{k-1}+2.\label{ineq-hahby2}
\end{align}
\end{claim}
\begin{proof}
Since $2\leq q\leq k-1$, by Claim \ref{claim2} we see that $\hb(Y)$ is non-trivial. Note that $\hb(Y)$ is $(k-q)$-uniform and $\ha(\overline{Y})$ is $(k-1)$-uniform. If $\ha(\overline{Y})\neq \emptyset$, then by \eqref{hahbnontrivial}
\[
|\ha(\overline{Y})|+|\hb(Y)| \leq   \binom{n-q-1}{k-1} - 2\binom{n-k-1}{k-1}+\binom{n-2k+q-1}{k-1}+2.
\]
Moreover, the equality holds iff $\hb(Y)$ consists of two disjoint sets $S,T\in \binom{V\setminus Y}{k-q}$. Then $\hb=\{X, S\cup Y, T\cup Y\}$. By the intersection property of $\hb$, there exist $x_i\in S\cap X$ and $x_j\in T\cap X$. But then $\{x_i,x_j\}$ is a 2-cover of $\hb$, contradicting $\hhat=X\times Y$. Thus the inequality is strict and \eqref{ineq-hahby2} holds.

 If $\ha(\overline{Y})= \emptyset$, since $\hb$ is intersecting, then
\begin{align*}
|\ha(\overline{Y})|+|\hb(Y)| = |\hb(Y)|& \leq\binom{n-q-1}{k-q} -\binom{n-k-q-1}{k-q}\\[3pt]
&\overset{\eqref{ineq-key9}}<\binom{n-q-1}{k-1} - 2\binom{n-k-1}{k-1}+\binom{n-2k+q-1}{k-1}+2.
\end{align*}
\end{proof}

By adding \eqref{ineq-ha0} with $p=k$, \eqref{ineq-hahbx2} and \eqref{ineq-hahby2}, we arrive at
\begin{align*}
|\ha|+|\hb|&\leq |\ha_0|+|\ha(\overline{X})|+|\hb(X)|+|\ha(\overline{Y})|+|\hb(Y)|\\[3pt]
&< \binom{n-1}{k-1}+\binom{n-k-q-1}{k-1}-3\binom{n-k-1}{k-1}+\binom{n-2k+q-1}{k-1}+3.
\end{align*}
Let
\[
f(q)= \binom{n-1}{k-1}+\binom{n-k-q-1}{k-1}-3\binom{n-k-1}{k-1}+\binom{n-2k+q-1}{k-1}+3.
\]
Note that $2f(q)< f(q+1)+f(q-1)$ for $2< q< k-1$. We infer $|\ha|+|\hb| < \max\{f(2),f(k-1)\}$.
Let us show  that for $k\geq 4$, $f(2)<|\hg(n,k)|$.
\begin{align*}
f(2)=& \binom{n-1}{k-1}+\binom{n-k-3}{k-1}-3\binom{n-k-1}{k-1}+\binom{n-2k+1}{k-1}+3\\[3pt]
=&\binom{n-1}{k-1}-\binom{n-k-1}{k-1}-\binom{n-k}{k-1}+\binom{n-k-1}{k-2}+\binom{n-k-3}{k-1}
-\binom{n-k-1}{k-1}\\[3pt]
&\qquad +\binom{n-2k}{k-1}+\binom{n-2k}{k-2}+3\\[3pt]
=&|\hg(n,k)|+\binom{n-k-1}{k-2}+\binom{n-k-3}{k-1}-\binom{n-k-1}{k-1}+\binom{n-2k}{k-2}-\binom{n-k-2}{k-3}\\[3pt]
=&|\hg(n,k)|+\binom{n-2k}{k-2}-\binom{n-k-3}{k-2}<|\hg(n,k)|.
\end{align*}
Let us show next $f(k-1)= |\hg(n,k)|$.
\begin{align*}
f(k-1)=& \binom{n-1}{k-1}+\binom{n-2k}{k-1}-3\binom{n-k-1}{k-1}+\binom{n-k-2}{k-1}+3\\[3pt]
=&\binom{n-1}{k-1}+\binom{n-2k}{k-1}-\binom{n-k-1}{k-1}-\binom{n-k}{k-1}+\binom{n-k-1}{k-2}
-\binom{n-k-1}{k-1}\\[3pt]
&\qquad+\binom{n-k-2}{k-1}+3\\[3pt]
=&|\hg(n,k)|+\binom{n-k-1}{k-2}-\binom{n-k-1}{k-1}+\binom{n-k-2}{k-1}-\binom{n-k-2}{k-3}\\[3pt]
=&|\hg(n,k)|+\binom{n-k-2}{k-2}-\binom{n-k-1}{k-1}+\binom{n-k-2}{k-1}=|\hg(n,k)|.
\end{align*}
Thus we obtain that $|\ha|+|\hb|< |\hg(n,k)|$ and the proposition is proven.
\end{proof}

Combining Propositions \ref{lem7.5} and \ref{lem7.6}, we conclude that the theorem holds.
\end{proof}

\vspace{6pt}\noindent
{\bf Acknowledgement:}  The first author's research was partially supported by the National Research, Development and Innovation Office NKFIH, grant K132696.

%
%
%

\end{document}